\documentclass[10pt,a4paper]{article}
\usepackage{amsfonts}
\usepackage{color}
\usepackage{latexsym,amsthm,amssymb,amscd}
\usepackage{amsmath}

\newcounter{alphthm}
\DeclareMathAlphabet{\mathpzc}{OT1}{pzc}{m}{it}
\newtheorem{theorem}{Theorem}[section]
\newtheorem{lemma}[theorem]{Lemma}
\newtheorem{proposition}[theorem]{Proposition}
\newtheorem{cor}[theorem]{Corollary}
\theoremstyle{definition}
\newtheorem{defn}[theorem]{Definition}%
\newtheorem{rem}[theorem]{Remark}
\newtheorem{remark}[theorem]{Remark}
\newtheorem{definition}[theorem]{Definition}
\newtheorem{example}[theorem]{Example}

\begin{document}

\title{Weyl's Theory in the Generalized Lie Algebroids Framework}
\author{C. M. Arcu\c{s}, E. Peyghan and E. Sharahi}
\maketitle

\begin{abstract}
The geometry of the Lie algebroid generalized tangent bundle of a
generalized Lie algebroid is developed. Formulas of Ricci type and
identities of Cartan and Bianchi type are presented. Introducing the
notion of geodesic of a mechanical $\left( \rho ,\eta \right)
$-system with respect to a $(\rho, \eta)$-spray, the Berwald $(\rho,
\eta)$-derivative operator and its mixed curvature, we obtain main
results to conceptualize the Weyl's method in this general
framework. Finally, we obtain two new results of Weyl type for the
geometry of mechanical $\left( \rho ,\eta \right) $-systems.
\end{abstract}

\section{Introduction}
The classical theory of physical geometry, developed by Helmholtz, Poincar%
\'{e} and Hilbert, regarded the concept of the \textquotedblleft
metric congruence\textquotedblright\ as the only basic relation in
geometry, and constructed physical geometry from this one notion
alone in terms of the relative positions and displacements of
physical congruence standards. Although Einstein's general theory of
relativity championed a dynamical view of spacetime geometry that is
very different from the classical theory of physical geometry,
Einstein initially approached the problem to find a structure of
spacetime from the metrical point of view. The decisive step in this
development, however, was T. Levi-Civita's discovery in 1917 of the
concept of infinitesimal parallel vector displacement, and the fact
that such parallel vector displacement is uniquely determined by the
metric field of Riemannian geometry. In 1918, Weyl presented his
ingenious attempt to unify gravitation and electromagnetism by
constructing a gauge-invariant geometry which is a purely
infinitesimal metric geometry \cite{31}. Weyl wanted a metric
geometry which would not permit distance comparison of lenght
between two vectors located at finitely different points. Weyl's
metric-independent construction of the affine structure led to the
development the differential projective geometry. The interest in
projective geometry is in the paths of curves.

It is important to understand that what one means by a
\textquotedblleft curve\textquotedblright\ is the map (the
parametric description) itself, and not the set of its image
points(namely, the path). Consequently, two curves are
mathematically considered to be different curves if they are given
by different maps (different parameter descriptions), even though
their image sets (namely, their paths) are the same \cite{9}.

Weyl recognized the importance of a torsion-free covariant
derivative operator independently of a metric (see \cite{30}) and in
1921 he proposed to extend the geometric structure (more precisely
the affine connection) determined by the system of freely falling
particles \cite{32}.

In a modern formulation, the Weyl's theorem turns up as follows:
\begin{flushleft}
\textit{Two torsion-free covariant derivative
operators $D_{M}$ and $\bar{D}_{M}$ have the same geodesics
up to a parameter transformation if and only if there exists a $1$-form $%
\alpha $ such that
\[
\bar{D}_{M}=D_{M}+\alpha \otimes 1+1\otimes \alpha ,
\]%
where $1$ is the identity tensor of $\left( 1,1\right)$-type.}

\textit{The covariant derivative operators $\bar{D}_{M}$ and $D_{M}$ are called
projectively equivalent.}
\end{flushleft}
A classical proof of Weyl's theorem is presented in \cite{12} and some
 coordinate-free proofs are presented in e.g. \cite{16,18,22}. As
there exists a canonical bijective corespondence between the set of
torsion-free covariant derivative of tensorial algebra of the usual Lie
algebroid $\left( \left( TM,\tau _{M},M\right) ,\left[ ,\right] _{TM},\left(
Id_{TM},Id_{M}\right) \right) $ and the set of sprays of the usual Lie
algebroid $\left( \left( TTM,\tau _{TM},TM\right) ,\left[ ,\right]
_{TTM},\left( Id_{TTM},Id_{TM}\right) \right) ,$ a first modern formulation
of Weyl's theorem in this new direction by research was presented by
Ambrose, Palais and Singer \cite{1}. See also \cite{8}.

Using the pull-back formalism given by the short exact sequence
\[
\begin{array}{ccccccccc}
0 & \hookrightarrow  & \tau _{M}^{\ast }\left( TM\right)  & ^{%
\underrightarrow{~\ i~\ }} & TTM & ^{\underrightarrow{~\ j~\ }} & \tau
_{M}^{\ast }\left( TM\right)  & ^{\underrightarrow{}} & 0 \\
\downarrow  &  & \downarrow  &  & \downarrow  &  & \downarrow  &  &
\downarrow  \\
TM & ^{\underrightarrow{Id_{TM}~}} & TM & ^{\underrightarrow{~Id_{TM}~}} &
TM & ^{\underrightarrow{~Id_{TM}~}} & TM & ^{\underrightarrow{~Id_{TM}~}} &
TM%
\end{array}%
\]
which is more known in Finsler geometry (see e.g. \cite{6,23}), in
the geometry of second-order ordinary differential equations or in
the inverse problem of calculus of variations \cite{7,14},  Szilasi
 examines in  \cite{24} how some basic geometric data will transform
under a projective change of the spray.
 In addition, using the Yano derivative operator (see \cite{33}),
Szilasi extend the Weyl's theorem to the general class of non-affine sprays.

Using the prolongation Lie algebroid
\[
\left( \left( \mathcal{L}^{\pi }E,\pi _{\mathcal{L}},E\right) ,\left[ ,%
\right] _{\mathcal{L}},\left( \rho _{\mathcal{L}},Id_{E}\right) \right),
\]%
of a Lie algebroid
\[
\left( \left( E,\pi ,M\right) ,\left[ ,\right] _{E},\left( \rho
,Id_{M}\right) \right),
\]%
the pull-back formalism given by the short exact sequence
\[
\begin{array}{ccccccccc}
0 & \hookrightarrow  & \pi ^{\ast }\left( E\right)  & ^{\underrightarrow{~\
i~\ }} & \mathcal{L}^{\pi }E & ^{\underrightarrow{~\ j~\ }} & \pi ^{\ast
}\left( E\right)  & ^{\underrightarrow{}} & 0 \\
\downarrow  &  & \downarrow  &  & \downarrow  &  & \downarrow  &  &
\downarrow  \\
E & ^{\underrightarrow{~Id_{E}~}} & E & ^{\underrightarrow{~Id_{E}~}} & E &
^{\underrightarrow{~Id_{E}~}} & E & ^{\underrightarrow{~Id_{E}~}} & E%
\end{array}%
\]%
is studied by many e.g. \cite{10,11,13,14,15} and similar results
with the classical ones were appeared (see also
\cite{19,20,21,25,26,27,28,29}).

Recently, using the previous short exact sequence, the second author rebuild the Finsler geometry concepts on the Lie algebroid structures and obtained important
 results about Liouville section and (semi)sprays \cite{17}.

The first author in \cite{2} introduced the Lie algebroid generalized tangent
bundle%
\[
\left( (\left( \rho ,\eta \right) TE,\left( \rho ,\eta \right) \tau _{E},E),
\left[ ,\right] _{\left( \rho ,\eta \right) TE},\left( \tilde{\rho}%
,Id_{E}\right) \right) ,
\]%
of a generalized Lie algebroid
\[
\left( \left( E,\pi ,M\right) ,\left[ ,\right] _{E,h},\left( \rho ,\eta
\right) \right),
\]%
and developed the geometry of mechanical $\left( \rho ,\eta \right) $%
-systems in \cite{4} using the short exact sequence%
\[
\begin{array}{ccccccccc}
0 & \hookrightarrow  & V\left( \rho ,\eta \right) TE & \hookrightarrow  &
\left( \rho ,\eta \right) TE & ^{\underrightarrow{~\ \left( \rho ,\eta
\right) \pi !~\ }} & \left( h\circ \pi \right) ^{\ast }E & ^{%
\underrightarrow{}} & 0 \\
\downarrow  &  & \downarrow  &  & \downarrow  &  & \downarrow  &  &
\downarrow  \\
E & ^{\underrightarrow{~Id_{E}~}} & E & ^{\underrightarrow{~Id_{E}~}} & E &
^{\underrightarrow{~Id_{E}~}} & E & ^{\underrightarrow{~Id_{E}~}} & E%
\end{array}%
\]

During study of sprays and geodesics in the general mechanical $(\rho ,\eta)$-systems, a natural question will be arised.
\begin{quote}
{\small Can we obtain theorems of Weyl type in the general framework
of mechanical }$(\rho ,\eta)${\small -systems?}
\end{quote}
We answer to the question in the last section. To this end, we must introduce principle concepts like the geodesic as the heart of the problem,
 the Berwald derivative operator and its mixed curvature. These are the main tools to conceptualize the Weyl method.

This paper is arranged as follows. In section 2, presenting a
(generalized) Lie algebroid, the Lie algebroid generalized tangent
bundle is introduced. In subsection 2.1 connections on a total space
of a generalized Lie algebroid are studied.
 For more details, one can see also \cite{3}. In subsection 2.2, we introduce a lift of a curve on the base of a generalized Lie algebroid that
  leads us to derive the differential equation for geodesics.
 In subsection 2.3 some classical endomorphisms for the module of sections of the Lie algebroid generalized tangent bundle are made.
  Section 6 is devoted to developing the theory of general distinguished linear connections for the Lie algebroid generalized tangent bundle.
   The torsion and the curvature of a distinguished linear $\left( \rho ,\eta\right) $-connection is developed and studied next.
    Formulas of Ricci type and identities of Cartan and Bianchi type in the general framework of the Lie algebroid generalized tangent bundle of a
     generalized Lie algebroid end this section. A presentation of the $\left( \rho ,\eta \right) $-semisprays and $%
\left( \rho ,\eta \right) $-sprays is in section 4. Other aspects
and more details can be found in \cite{4}. Also, there is a theorem
in this section
 that classifies the integral curves inducing by a $(\rho,\eta)$-spray. The section 5 is reserved to two theorems of the Weyl type.
  In this section, beginning with a Berwald covariant derivative, we study the horizontal projector associated to a $( \rho ,\eta ) $-spray.
   Moreover, we develop the sense of projectively related of $( \rho ,\eta ) $-sprays using the Liouville section. Then, defining the concept of geodesic of
    a mechanical $(\rho, \eta)$-system with respect to a $(\rho, \eta)$-spray, we answer to the above mentioned question in terms of two final theorems.
\section{The Lie algebroid generalized tangent bundle of a generalized Lie algebroid}
Let $(F,\nu ,N)$ be a vector bundle, $\Gamma ( F, \nu, N)$ be the set of the sections of it and $\mathcal{F}(
N)$ be the set of smooth real-valued functions on $N$. Then $(\Gamma
(F,\nu ,N) ,+,\cdot)$ is a $\mathcal{F}(N)$-module. If $(\varphi ,\varphi _{0})$ is a morphism from $( F,\nu ,N)$ to $(F', \nu', N')$ such that $\varphi _{0}$ is a isomorphism from $N$ to $N'$, then using the operation
\begin{equation*}
\begin{array}{ccc}
\mathcal{F}\left( N\right) \times \Gamma \left(F', \nu', N'\right) & ^{\underrightarrow{~\ \ \cdot ~\ \ }} & \Gamma \left(
F', \nu', N'\right),\\
\left( f,u^{\prime }\right) & \longmapsto & f\circ \varphi _{0}^{-1}\cdot
u^{\prime },
\end{array}%
\end{equation*}%
it results that $(\Gamma( F', \nu', N') ,+,\cdot)$ is a $\mathcal{F}(N)$-module and the modules morphism
\begin{equation*}
\begin{array}{ccc}
\Gamma \left(F,\nu ,N\right) & ^{\underrightarrow{~\ \ \Gamma \left(
\varphi ,\varphi _{0}\right) ~\ \ }} & \Gamma \left( F', \nu', N'\right),\\
u & \longmapsto & \Gamma \left( \varphi ,\varphi _{0}\right) u,
\end{array}%
\end{equation*}%
defined by
\begin{equation*}
\begin{array}{c}
\Gamma ( \varphi ,\varphi _{0}) u(y) =\varphi ( u_{\varphi
_{0}^{-1}(y)}) =( \varphi \circ u\circ
\varphi _{0}^{-1})(y),%
\end{array}%
\end{equation*}%
for any $y\in N'$ will be obtained.
\begin{definition}\label{AP}
A generalized Lie algebroid is a vector bundle $(F,\nu ,N)$ given by the diagram
\begin{equation}
\begin{array}{c}
\begin{array}[b]{ccccc}
\left( F,\left[ ,\right] _{F,h}\right) & ^{\underrightarrow{~\ \ \ \rho \ \
\ \ }} & \left( TM,\left[ ,\right] _{TM}\right) & ^{\underrightarrow{~\ \ \
Th\ \ \ \ }} & \left( TN,\left[ ,\right] _{TN}\right)\\
~\downarrow \nu &  & ~\ \ \downarrow \tau _{M} &  & ~\ \ \downarrow \tau _{N}
\\
N & ^{\underrightarrow{~\ \ \ \eta ~\ \ }} & M & ^{\underrightarrow{~\ \ \
h~\ \ }} & N
\end{array}
\end{array}
\end{equation}
where $h$ and $\eta $ are arbitrary isomorphisms, $(\rho, \eta): (F,\nu, N)\longrightarrow (TM,\tau _{M},M)$ is a vector bundles morphism and
\begin{equation*}
\begin{array}{ccc}
\Gamma \left( F,\nu ,N\right) \times \Gamma \left( F,\nu ,N\right) & ^{%
\underrightarrow{~\ \ \left[ ,\right] _{F,h}~\ \ }} & \Gamma \left( F,\nu
,N\right),\\
\left( u,v\right) & \longmapsto & \ \left[ u,v\right] _{F,h},
\end{array}
\end{equation*}
is an operation satisfies in
\begin{equation*}
\begin{array}{c}
\left[ u,f\cdot v\right] _{F,h}=f\left[ u,v\right] _{F,h}+\Gamma \left(
Th\circ \rho ,h\circ \eta \right) \left( u\right) f\cdot v,\ \ \ \forall f\in \mathcal{F}(N),
\end{array}
\end{equation*}
such that the 4-tuple $(\Gamma( F, \nu, N) ,+,\cdot, [ , ] _{F,h})$ is a Lie $\mathcal{F}(N)$-algebra.
\end{definition}
We denote by $\Big((F, \nu, N), [ , ] _{F,h}, (\rho, \eta) \Big)$ the generalized Lie algebroid defined in the above. Moreover, the couple $\Big([ , ]
_{F,h}, (\rho, \eta)\Big)$ is called the \emph{generalized
Lie algebroid structure}. It is easy to see that in the above definition, the modules morphism $\Gamma(Th\circ\rho, h\circ \eta) : (\Gamma(F,\nu ,N), +, \cdot, [ , ] _{F,h})\longrightarrow(\Gamma(TN, \tau_{N}, N), +, \cdot, [ , ]_{TN})$ is a Lie algebras morphism.
\begin{definition}
A morphism from
$
(( F,\nu ,N), [ , ] _{F,h}, (\rho, \eta))
$
to\\
$
(( F^{\prime }, \nu ^{\prime }, N^{\prime }), [ , ]
_{F^{\prime }, h^{\prime }}, (\rho ^{\prime }, \eta ^{\prime }))
$
is a morphism $(\varphi ,\varphi _{0})$ from $(F,\nu, N)$ to $(F',\nu', N')$ such that $\varphi _{0}$ is an isomorphism from $N$ to $N'$, and the modules morphism $\Gamma(\varphi, \varphi _{0})$
is a Lie algebras morphism from
$
( \Gamma \left( F,\nu ,N\right) ,+,\cdot ,\left[ ,\right] _{F,h})
$
to\\
$
( \Gamma \left( F^{\prime },\nu ^{\prime },N^{\prime }\right) ,+,\cdot
, \left[ ,\right] _{F^{\prime },h^{\prime }}).
$
\end{definition}
\begin{remark}
In particular case $\left( \eta ,h\right) =\left(
Id_{M},Id_{M}\right)$, the definition of the Lie algebroid will be obtained.
\end{remark}
In the next we consider that $(\left( E,\pi ,M\right) ,\left[
,\right] _{E,h},\left( \rho ,\eta \right))$ is a generalized lie
algebroid, where $\eta$ and $h$ are diffeomorphisms on $M$.
 Setting $\left( x^{i},y^{a}\right) $ as the canonical local coordinates on $%
\left( E,\pi ,M\right) ,$ where $i\in 1,\cdots ,m$, $a\in 1,\cdots
,r$, and
\begin{equation*}
\left( x^{i},y^{a}\right) \longrightarrow \left( x^{i%
{\acute{}}%
}\left( x^{i}\right) ,y^{a%
{\acute{}}%
}\left( x^{i},y^{a}\right) \right) ,
\end{equation*}%
is a change of coordinates on $\left( E,\pi ,M\right) $, then the
coordinates $y^{a}$ change to $y^{a%
{\acute{}}%
}$ according to the rule
\begin{equation}
\begin{array}{c}
y^{a%
{\acute{}}%
}=\displaystyle M_{a}^{a^{\prime }}y^{a}.%
\end{array}%
\end{equation}%
We consider the pull-back bundle $\left( \left( h\circ \pi \right)
^{\ast }E,\left( h\circ \pi \right) ^{\ast }\pi ,E\right) $ of this
generalized Lie algebroid. If we define
\begin{equation}
\begin{array}{rcl}
\ \left( h\circ \pi \right) ^{\ast }E & ^{\underrightarrow{\overset{\left(
h\circ \pi \right) ^{\ast }E}{\rho }}} & TE \\
\displaystyle X^{a}S_{a}(u_{x}) & \longmapsto & \displaystyle
X^{a}\left( \rho _{a}^{i}\circ h\circ \pi \right) \frac{\partial
}{\partial x^{i}}\left( u_{x}\right),
\end{array}
\end{equation}%
then $(({\overset{\left( h\circ \pi \right) ^{\ast }E}{\rho }},Id_{E})) : \left( \left( h\circ \pi \right) ^{\ast }E,\left( h\circ \pi \right) ^{\ast
}\pi ,E\right) \longrightarrow\left( TE,\tau _{E},E\right)$ is a vector bundles morphism. Moreover, the operation
\begin{equation*}
\begin{array}{ccc}
\Gamma \left( \left( h\circ \pi \right) ^{\ast }E,\left( h\circ \pi \right)
^{\ast }\pi ,E\right) ^{2} & ^{\underrightarrow{~\ \ \left[ ,\right]
_{\left( h\circ \pi \right) ^{\ast }E}~\ \ }} & \Gamma \left( \left( h\circ
\pi \right) ^{\ast }E,\left( h\circ \pi \right) ^{\ast }\pi ,E\right) ,%
\end{array}%
\end{equation*}%
defined by
\begin{equation}
\begin{array}{ll}
\left[ S_{a},S_{b}\right] _{\left( h\circ \pi \right) ^{\ast }E} &
=(L_{ab}^{c}\circ h\circ \pi )S_{c},\vspace*{1mm} \\
\left[ S_{a},fS_{b}\right] _{\left( h\circ \pi \right) ^{\ast }E} & %
\displaystyle=f\left( L_{ab}^{c}\circ h\circ \pi \right) S_{c}+\left( \rho
_{a}^{i}\circ h\circ \pi \right) \frac{\partial f}{\partial x^{i}}S_{b},%
\vspace*{1mm} \\
\left[ fS_{a},S_{b}\right] _{\left( h\circ \pi \right) ^{\ast }E} & =-\left[
S_{b},fS_{a}\right] _{\left( h\circ \pi \right) ^{\ast }E},%
\end{array}%
\end{equation}%
for any $f\in \mathcal{F}\left( E\right)$, is a Lie bracket on $\Gamma \left( \left( h\circ \pi \right) ^{\ast }E,\left( h\circ \pi \right)
^{\ast }\pi ,E\right)$. It is easy to check that
\begin{equation*}
\Big( (\left( h\circ \pi \right) ^{\ast }E,\left( h\circ \pi \right) ^{\ast
}\pi ,E) ,\left[ ,\right] _{\left( h\circ \pi \right) ^{\ast }E},(\overset{%
\left( h\circ \pi \right) ^{\ast }E}{\rho },Id_{E})\Big),
\end{equation*}%
is a Lie algebroid which is called the \textit{pull-back Lie algebroid of
the generalized Lie algebroid} $( \left( E,\pi ,M\right) ,\left[ ,\right]
_{E,h},\left( \rho ,\eta \right))$.
\begin{remark}
If $u=u^{a}s_{a}$ is a section of $( E,\pi ,M)$, then its
corresponding section is
\begin{equation*}
X=X^{a}S_{a}\in\Gamma(( h\circ \pi ) ^{\ast }E, (h\circ \pi ) ^{\ast
}\pi, E),
\end{equation*}
given by $X\left( u_{x}\right)=u\left( h\left( x\right) \right)$, for any $%
u_{x}\in \pi ^{-1}(V{\cap h}^{-1}W)$, where $(V, t_V)$ and $(W,
t_W)$ are two vector local $(m+r)$-charts such that $V{\cap
h}^{-1}W\neq\emptyset$.
\end{remark}
Let $( \partial _{i},\dot{\partial}_{a}) $ be the base sections for the Lie $%
\mathcal{F}( E) $-algebra
\[
(\Gamma ( TE,\tau _{E},E) ,+, \cdot ,[ ,]
_{TE}).
\]
Then
\begin{align*}
X^{a}\tilde{\partial}_{a}+\tilde{X}^{a}\dot{\tilde{\partial}}_{a}&
:=X^{a}(S_{a}\oplus (\rho _{a}^{i}\circ h\circ \pi )\partial
_{i})+\tilde{X}^{a}(0_{\pi ^{\ast }E}\oplus \dot{\partial}_{a}) \\
& =X^{a}S_{a}\oplus (X^{a}(\rho _{a}^{i}\circ h\circ \pi )\partial _{i}+\tilde{X}^{a}%
\dot{\partial}_{a}),
\end{align*}
making from $X^{a}S_{a}\in \Gamma \left( \left( h\circ \pi \right)
^{\ast
}E,\left( h\circ \pi \right) ^{\ast }\pi ,E\right) $ and $\tilde{X}^{a}\dot{\partial}%
_{a}\in \Gamma \left( VTE,\tau _{E},E\right)$ is a section of
$(\left( h\circ \pi \right) ^{\ast }E\oplus TE,\overset{\oplus }{\pi
},E)$. Moreover,
it is easy to see that the sections $\tilde{\partial}_{1},\cdots,\tilde{%
\partial}_{r},\dot{\tilde{\partial}}_{1},\cdots,\dot{\tilde{\partial}}_{r}$
are linearly independent. Now, we consider the vector subbundle $(\left(
\rho ,\eta \right) TE,\left( \rho ,\eta \right) \tau _{E},E)$ of the vector
bundle $(\left( h\circ \pi \right) ^{\ast }E\oplus TE,\overset{\oplus }{\pi }%
,E),$ for which the $\mathcal{F}(E)$-module of sections is the $\mathcal{F}%
(E)$-submodule of $(\Gamma (\left( h\circ \pi \right) ^{\ast }E\oplus TE,%
\overset{\oplus }{\pi },E),+,\cdot ),$ generated by the set of sections $(%
\tilde{\partial}_{a},\dot{\tilde{\partial}}_{a}).$ The base sections $(%
\tilde{\partial}_{a},\dot{\tilde{\partial}}_{a})$ are called the \emph{%
natural }$(\rho ,\eta )$\emph{-base}.

Now consider  the vector bundles morphism $\left( \tilde{\rho},Id_{E}\right):( (\rho, \eta) TE, (\rho, \eta) \tau _{E} ,E)$ $ \longrightarrow( TE,\tau _{E},E) $
, where
\begin{equation}
\begin{array}{rcl}
\left( \rho ,\eta \right) TE\!\!\! & \!\!^{\underrightarrow{\ \ \tilde{\rho}%
\ \ }}\!\!\! & \!\!TE\vspace*{2mm} \\
(X^{a}\tilde{\partial}_{a}+\tilde{X}^{a}\dot{\tilde{\partial}}_{a})\!(u_{x})\!\!\!\!
& \!\!\longmapsto \!\!\! & \!\!(\!X^{a}(\rho _{a}^{i}{\circ h\circ
}\pi
)\partial _{i}{+}\tilde{X}^{a}\dot{\partial}_{a})\!(u_{x}).%
\end{array}%
\end{equation}%
Moreover, define the Lie bracket $[, ]_{( \rho ,\eta) TE}$ as follows
\begin{align}
\Big[X_{1}^{a}\tilde{\partial}_{a}+\tilde{X}_{1}^{a}\dot{\tilde{\partial}}%
_{a},X_{2}^{b}\tilde{\partial}_{b}+\tilde{X}_{2}^{b}\dot{\tilde{\partial}}_{b}\Big]%
_{\left( \rho ,\eta \right) TE}& =\Big[X_{1}^{a}S_{a}, X_{2}^{b}S_{b}\Big]%
_{\pi ^{\ast }E}\oplus \Big[X_{1}^{a}(\rho _{a}^{i}\circ h\circ \pi )%
\partial _{i}  \notag \\
& +\tilde{X}_{1}^{a}\dot{\partial}_{a},\Big.\hfill \displaystyle\Big.X_{2}^{b}(%
\rho _{b}^{j}\circ h\circ \pi )\partial _{j}+\tilde{X}_{2}^{b}\dot{\partial}_{b}%
\Big]_{TE}.
\end{align}%
Easily we obtain that $( \left[ ,\right] _{\left( \rho ,\eta \right) TE},\left( \tilde{%
\rho},Id_{E}\right)) $ is a Lie algebroid structure for the vector
bundle $\left( \left( \rho ,\eta \right) TE,\left( \rho ,\eta \right) \tau
_{E},E\right)$ which is called the generalized tangent bundle (see \cite{2, 3, 4}).
\subsection{$\left( \protect\rho ,\protect\eta \right) $-connections}
Consider the vector bundles morphism $\left( \left( \rho ,\eta \right)
\pi !,Id_{E}\right) $ given by the commutative diagram
\begin{equation}
\begin{array}{rcl}
\left( \rho ,\eta \right) TE & ^{\underrightarrow{~\ \left( \rho ,\eta
\right) \pi !~\ }} & (h\circ\pi)^{\ast }E \\
\left( \rho ,\eta \right) \tau _{E}\downarrow ~ &  & ~\downarrow(h\circ\pi)^{\ast
}\pi\\
E~\  & ^{\underrightarrow{~Id_{E}~}} & ~\ E%
\end{array}%
\end{equation}
by the rule
\begin{equation}
\begin{array}{c}
\left( \rho ,\eta \right) \pi !((X^{a}\tilde{\partial}_{a}+\tilde{X}^{a}\overset{%
\cdot }{\tilde{\partial}}_{a})(u_{x}))=(X^{a}S_{a})(u_{x}),%
\end{array}%
\end{equation}%
for any
$X^{a}\tilde{\partial}_{a}+\tilde{X}^{a}\dot{\tilde{\partial}}_{a}\in
\Gamma \left( \left( \rho ,\eta \right) TE,\left( \rho ,\eta \right)
\tau _{E},E\right)$.
Using the vector bundles morphism $(( \rho ,\eta)\pi!,Id_{E})$, the tangent $%
(\rho,\eta )$-application $((\rho,\eta ) T\pi ,\pi )$ from\\ $(( \rho ,\eta )
TE, ( \rho ,\eta ) \tau _E,E)$ to $( E,\pi ,M) $ will be obtained.

The kernel of the tangent $\left( \rho ,\eta \right) $-application
is denoted by $( V( \rho ,\eta ) TE,$ $( \rho
,\eta ) \tau _{E},E) $ and is called the\emph{
vertical subbundle}. Moreover, the set
$\{\dot{\tilde{\partial}}_{a}|~a\in 1,\cdots ,r\}$ is a base of the
$\mathcal{F}\left( E\right) $-module $\left( \Gamma \left( V\left(
\rho ,\eta \right) TE,\left( \rho ,\eta \right) \tau _{E},E\right)
,+,\cdot \right) $.

\begin{proposition}
The short sequence of vector bundles
\begin{equation}
\begin{array}{ccccccccc}
0 & \hookrightarrow & V\left( \rho ,\eta \right) TE & \hookrightarrow &
\left( \rho ,\eta \right) TE & ^{\underrightarrow{~\ \left( \rho ,\eta
\right) \pi !~\ }} & \left( h\circ \pi \right) ^{\ast }E & ^{%
\underrightarrow{}} & 0 \\
\downarrow &  & \downarrow &  & \downarrow &  & \downarrow &  & \downarrow
\\
E & ^{\underrightarrow{~Id_{E}~}} & E & ^{\underrightarrow{~Id_{E}~}} & E &
^{\underrightarrow{~Id_{E}~}} & E & ^{\underrightarrow{~Id_{E}~}} & E%
\end{array}
\label{3}
\end{equation}%
is exact.
\end{proposition}

A manifolds morphism $\left( \rho ,\eta \right) \Gamma $ from $\left( \rho
,\eta \right) TE$ to $V\left( \rho ,\eta \right) TE$ defined by
\begin{equation}
\begin{array}{c}
\left( \rho ,\eta \right) \Gamma
(X^{c}\tilde{\partial}_{c}+\tilde{X}^{a}\dot{\tilde{\partial}}_{a})(u_{x})=\left(
\tilde{X}^{a}+\left( \rho ,\eta \right)
\Gamma _{c}^{a}X^{c}\right) \dot{\tilde{\partial}}_{a}\left(u_{x}\right) ,%
\end{array}%
\end{equation}%
so that the vector bundles morphism $\left( \left( \rho ,\eta
\right) \Gamma ,Id_{E}\right) $ is a split to the left in the exact
sequence (\ref{3}), will be called $\left( \rho ,\eta \right)
$\emph{-connection for the vector bundle }$\left( E,\pi ,M\right) $.
If $\left( \rho ,\eta \right) \Gamma $ is a $\left( \rho ,\eta
\right) $-connection for the vector bundle $\left( E,\pi ,M\right)
$, then the kernel of the vector bundles morphism $\left( \left(
\rho ,\eta \right) \Gamma ,Id_{E}\right) $\ denoting by $\left(
H\left( \rho ,\eta \right) TE,\left( \rho ,\eta \right) \tau
_{E},E\right) $ is called the \emph{horizontal vector subbundle}.
Putting
\begin{equation}
\begin{array}[t]{l}
\tilde{\delta}_{a}=\tilde{\partial}_{a}-\left( \rho ,\eta \right)
\Gamma _{a}^{b}\dot{\tilde{\partial}}_{b}=S_{a}\oplus ((\rho
_{a}^{i}\circ h\circ
\pi )\partial _{i}-\left( \rho ,\eta \right) \Gamma _{a}^{b}\dot{\partial}%
_{b}),
\end{array}
\end{equation}
it is easy to see that $\{\tilde{\delta}_{a}|a=1,\cdots, r\}$ is a
base of the $\mathcal{F}(E)$ module
\begin{equation*}
\left( \Gamma \left( H\left( \rho ,\eta \right) TE,\left( \rho ,\eta
\right) \tau _{E},E\right) ,+,\cdot \right).
\end{equation*}%
The base $(\tilde{\delta}_{a},\dot{\tilde{\partial}}_{a})$ will be
called the \emph{adapted }$\left( \rho ,\eta \right) $\emph{-base.}
Let $(d\tilde{x}^{a},d\tilde{y}^{b})$ be the natural dual $\left(
\rho ,\eta
\right) $-base of natural $\left( \rho ,\eta \right) $-base $(\displaystyle%
\tilde{\partial}_{a},\displaystyle\dot{\tilde{\partial}}_{a})$. Then $\left( d\tilde{x}^{a},\delta \tilde{y}%
^{a}\right) $ is the adapted dual $\left( \rho ,\eta \right)$-base
of $(\tilde{\delta}_{a},\dot{\tilde{\partial}}_{a})$, where
\begin{equation}
\begin{array}{l}
\delta \tilde{y}^{a}=\left( \rho ,\eta \right) \Gamma _{c}^{a}d\tilde{x}%
^{c}+d\tilde{y}^{a},\ \ \ a\in 1,\cdots ,r.
\end{array}%
\end{equation}%
This base is called the \emph{adapted dual }$\left( \rho ,\eta \right) $%
\emph{-base}. One can deduce that 
\begin{equation}
\begin{array}{l}
\Gamma (\tilde{\rho},Id_{E})(\tilde{\delta}_{a})=(\rho _{a}^{i}\circ h\circ
\pi )\partial _{i}-\left( \rho ,\eta \right) \Gamma _{a}^{b}\dot{\partial}%
_{b},%
\end{array}%
\end{equation}%
where $(\partial _{i},\dot{\partial}_{a})$ is the natural base for the $%
\mathcal{F}(E)$-module $\left( \Gamma \left( TE,\tau _{E},E\right) ,+,\cdot
\right)$ \cite{2}.
\begin{theorem}
The equalities
\begin{align}
\lbrack \tilde{\delta}_{a},\tilde{\delta}_{b}]_{\left( \rho ,\eta \right)
TE}&=L_{ab}^{c}\circ h\circ \pi \cdot \tilde{\delta}_{c}+\left( \rho ,\eta
,h\right) \mathbb{R}_{\,\ ab}^{c}\dot{\tilde{\partial}}_{c},\\
\lbrack \tilde{\delta}_{a},\dot{\tilde{\partial}}_{b}]_{\left( \rho
,\eta \right) TE}&=\Gamma
(\tilde{\rho},Id_{E})(\dot{\tilde{\partial}}_{b})(\left(
\rho ,\eta \right) \Gamma _{a}^{c})\dot{\tilde{\partial}}_{c},\\
\lbrack
\dot{\tilde{\partial}}_{a},\dot{\tilde{\partial}}_{b}]_{\left( \rho
,\eta \right) TE}&=0,
\end{align}
where
\begin{align}
\left( \rho ,\eta ,h\right) \mathbb{R}_{\,\ ab}^{c}&=\Gamma (\tilde{\rho}%
,Id_{E})(\tilde{\delta}_{b})(\left( \rho ,\eta \right) \Gamma _{a}^{c})
-\Gamma (\tilde{\rho},Id_{E})(\tilde{\delta}_{a})(\left( \rho
,\eta \right) \Gamma _{b}^{c})\nonumber\\
&\ \ \ +(L_{ab}^{d}\circ h\circ \pi )
\left( \rho ,\eta \right) \Gamma _{d}^{c}.
\end{align}
are hold. Moreover
\begin{equation}
\begin{array}{c}
\Gamma (\tilde{\rho},Id_{E})[\tilde{\delta}_{a},\tilde{\delta}_{b}]_{\left(
\rho ,\eta \right) TE}=[\Gamma (\tilde{\rho},Id_{E})(\tilde{\delta}%
_{a}),\Gamma (\tilde{\rho},Id_{E})(\tilde{\delta}_{b})]_{TE}.%
\end{array}%
\end{equation}
\end{theorem}
\subsection{The $\left( g,h\right) $-lift of a differentiable curve}
Let $c:I\rightarrow M$ be a differentiable curve. Then $$( E_{|Im( \eta \circ h\circ
c) },\pi _{|Im( \eta \circ h\circ c) }, Im( \eta \circ
h\circ c) ),$$ is a vector subbundle of the vector bundle $\left(
E,\pi ,M\right) $. If
\begin{equation*}
\begin{array}{ccc}
I & ^{\underrightarrow{\ \ \dot{c}\ \ }} & E_{|Im\left( \eta \circ h\circ
c\right) } \\
t & \longmapsto & y^{a}\left( t\right) s_{a}\left( \left( \eta \circ h\circ
c\right) \left( t\right) \right),
\end{array}%
\end{equation*}%
is a differentiable curve such that there exists a vector bundles morphism $%
\left( g,h\right) $ from $\left( E,\pi ,M\right) $ to $\left( E,\pi
,M\right) $ satisfying in
\begin{equation}
\begin{array}{c}
\rho \circ g\circ \dot{c}\left( t\right) =\displaystyle\frac{d\left(
\eta \circ h\circ c\right) ^{i}\left( t\right) }{dt}\frac{\partial
}{\partial x^{i}}\left( \left( \eta \circ h\circ c\right) \left(
t\right) \right),\label{17J}
\end{array}
\end{equation}%
for any $t\in I$, then $\dot{c}$ is called {\it the $\left( g,h\right)
$-lift of the differentiable curve $c$}. Moreover, the section%
\begin{equation}
\begin{array}{ccc}
Im\left( \eta \circ h\circ c\right) & ^{\underrightarrow{u\left( c,\dot{c}%
\right) }} & E_{|Im\left( \eta \circ h\circ c\right) }\vspace*{1mm} \\
\eta \circ h\circ c\left( t\right) & \longmapsto & \dot{c}\left( t\right),
\end{array}
\label{eq38}
\end{equation}%
will be called the\emph{\ canonical section associated to the couple }$%
\left( c,\dot{c}\right) .$

\begin{definition}
\label{d18}If the vector bundles morphism $\left( g,h\right) $ has the
components
\begin{equation*}
\begin{array}{c}
g_{b}^{a};~a,b\in {1, \cdots, r},
\end{array}%
\end{equation*}%
such that for any vector local $\left( m+r\right) $-chart $(V%
,t_V) $ of $\left( E,\pi ,M\right) $ there exists the real functions
\begin{equation*}
\begin{array}{ccc}
V & ^{\underrightarrow{~\ \ \ \tilde{g}_{a}^{b}~\ \ }} & \mathbb{R}%
\end{array}%
;~a,b=1,\cdots, r,
\end{equation*}%
such that%
\begin{equation}
\begin{array}{c}
\tilde{g}_{c}^{b}\left( x\right) \cdot g_{a}^{c}\left( x\right) =\delta
_{a}^{b},%
\end{array}
\label{eq39}
\end{equation}%
for any $x\in V,$ then we will say that the vector bundles
morphism $\left( g,h\right) $ is locally invertible.
\end{definition}
Using the components of $(g, h)$, the condition (\ref{17J}) is
equivalent to
\begin{equation}\label{God}
\begin{array}[b]{c}
\rho _{d}^{i}\left( \eta \circ h\circ c\left( t\right) \right)
g_{a}^{d}\left( h\circ c\left( t\right) \right) y^{a}\left( t\right) =%
\frac{d\left( \eta \circ h\circ c\right) ^{i}\left( t\right)
}{dt},\ \ \ i\in {1, \cdots, m}.%
\end{array}
\end{equation}
\begin{remark}
\label{r19}\textrm{In particular, if }$\left( Id_{TM},Id_{M},Id_{M}\right)
=\left( \rho ,\eta ,h\right) $\textrm{\ and the vector bundles morphism }$%
\left( g,Id_{M}\right) $\textrm{\ is locally invertible, then we have the
differentiable }$\left( g,Id_{M}\right) $\textrm{-lift}%
\begin{equation}
\begin{array}{ccl}
I & ^{\underrightarrow{\ \ \dot{c}\ \ }} & TM \\
t & \longmapsto & \displaystyle\tilde{g}_{j}^{i}\left( c\left( t\right)
\right) \frac{dc^{j}\left( t\right) }{dt}\frac{\partial }{\partial x^{i}}%
\left( c\left( t\right) \right) .%
\end{array}
\label{eq40}
\end{equation}
Moreover, if $g=Id_{TM}$, then we obtain the usual lift of
tangent vectors
\begin{equation}
\begin{array}{ccl}
I & ^{\underrightarrow{\ \ \dot{c}\ \ }} & TM\vspace*{1mm} \\
t & \longmapsto & \displaystyle\frac{dc^{i}\left( t\right) }{dt}\frac{%
\partial }{\partial x^{i}}\left( c\left( t\right) \right) .%
\end{array}
\label{eq41}
\end{equation}
\end{remark}
\begin{definition}
\label{d20}If $%
\begin{array}{ccl}
I & ^{\underrightarrow{\ \ \dot{c}\ \ }} & E_{|Im\left( \eta \circ h\circ
c\right) }%
\end{array}%
$ is a differentiable $\left( g,h\right) $-lift of differentiable
curve $c,$ such that its component functions $y^{a}$, $a\in
\{1,\cdots, n\}$, are solutions for the differentiable system of equations
\begin{equation}
\begin{array}[b]{c}
\frac{du^{a}}{dt}+\left( \rho ,\eta \right) \Gamma _{d}^{a}\circ u\left( c,%
\dot{c}\right) \circ \left( \eta \circ h\circ c\right) \cdot g_{b}^{d}\circ
h\circ c\cdot u^{b}=0,%
\end{array}
\label{eq42}
\end{equation}%
then we will say that the $\left( g,h\right) $-lift $\dot{c}$
is parallel with respect to the $\left( \rho ,\eta \right) $-connection $\left( \rho ,\eta \right)\Gamma$.
\begin{remark}\label{r21}
In particular, if $\left( \rho ,\eta ,h\right) =\left(
Id_{TM},Id_{M},Id_{M}\right) $ and the vector bundles morphism $%
\left( g,Id_{M}\right) $ is locally invertible, then the
differentiable $\left( g,Id_{M}\right) $-lift
\begin{equation}
\begin{array}{ccl}
I & ^{\underrightarrow{\ \ \dot{c}\ \ }} & TM\vspace*{1mm} \\
t & \longmapsto & \displaystyle( \tilde{g}_{j}^{i}\circ c\cdot \frac{%
dc^{j}}{dt}) \frac{\partial }{\partial x^{i}}\left( c\left( t\right)
\right),%
\end{array}
\label{eq43}
\end{equation}%
is parallel with respect to the connection $\Gamma $ if
the component functions
\begin{equation*}
\begin{array}[b]{c}
 \tilde{g}_{j}^{i}\circ c\cdot \frac{dc^{j}}{dt},\ \ \ i\in \{1, \cdots, n\},
\end{array}%
\end{equation*}%
are solutions for the differentiable system of equations
\begin{equation}
\begin{array}[b]{c}
\frac{du^{i}}{dt}+\Gamma _{k}^{i}\circ u\left( c,\dot{c}\right) \circ c\cdot
g_{h}^{k}\circ c\cdot u^{h}=0,%
\end{array}
\label{eq44}
\end{equation}%
namely
\begin{align}\label{eq45}
\frac{d}{dt}( \tilde{g}_{j}^{i}( c( t)
) \cdot \frac{dc^{j}( t) }{dt})+\Gamma _{k}^{i}( ( \tilde{g}_{j}^{i}(
c( t))\frac{dc^{j}( t) }{dt})\frac{\partial }{\partial x^{i}}( c( t) )
)\frac{dc^{k}( t) }{dt}=0.%
\end{align}
Moreover, if $g=Id_{TM}$, then the usual lift of the tangent
vectors \eqref{eq41} is parallel with respect to the connectoin  $\Gamma $
  if the component functions $
\frac{dc^{j}}{dt},\ j\in \{1,\cdots, n\}$, are
solutions for the differentiable system
of equations
\begin{equation}
\begin{array}[b]{c}
\frac{du^{i}}{dt}+\Gamma _{k}^{i}\circ u\left( c,\dot{c}\right) \circ c\cdot
u^{k}=0,%
\end{array}
\label{eq46}
\end{equation}%
\textrm{namely}%
\begin{equation}
\begin{array}[b]{c}
\frac{d}{dt}\left( \frac{dc^{j}\left( t\right) }{dt}\right) +\Gamma
_{k}^{i}\left( \frac{dc^{j}\left( t\right) }{dt}\cdot \frac{\partial }{%
\partial x^{i}}\left( c\left( t\right) \right) \right) \cdot \frac{%
dc^{k}\left( t\right) }{dt}=0.%
\end{array}
\label{eq47}
\end{equation}
\end{remark}
\end{definition}
\subsection{Remarkable modules endomorphisms}
In the next of the paper we present the locally expression of a
section
$X=X^a\tilde{\partial}_a+\tilde{X}^a\dot{\tilde{\partial}}_a$ of the
generalized tangent bundle $((\rho, \eta)TE,(\rho, \eta)\tau_{E},
E)$ with respect to the adapted $(\rho, \eta)$-base
$\{\tilde{\delta}_a, \dot{\tilde{\partial}}_a\}$ as
\[
X=X^a\tilde{\delta}_a+\dot{\tilde{X}}^a\dot{\tilde{\partial}}_a,
\]
where $\dot{\tilde{X}}^a=\tilde{X}^a+(\rho, \eta)\Gamma^a_bX^b$.
\begin{defn}
A modules endomorphism $e$ of $\Gamma (\left( \rho ,\eta \right)
TE,\!\left( \rho ,\eta \right) \tau _{E},\!E)$ with the property
$e^{2}=e$ will be called projector.
\end{defn}
The followings are two important examples of the projectors.
\begin{example}
The modules endomorphism
\begin{equation*}
\begin{array}{rcl}
\Gamma \!(\left( \rho ,\eta \right) TE,\!\left( \rho ,\eta \right) \tau
_{E},\!E) & ^{\underrightarrow{\ \ \mathcal{V}\ \ }} & \Gamma \!(\left( \rho
,\eta \right) TE,\!\left( \rho ,\eta \right) \tau _{E},\!E) \\
X^{a}\tilde{\delta}_{a}+\dot{\tilde{X}}^{a}\dot{\tilde{\partial}}_{a} & \longmapsto & \dot{\tilde{X}}^{a}%
\dot{\tilde{\partial}}_{a},
\end{array}%
\end{equation*}%
is a projector which is called the vertical projector. It is obvious that $\mathcal{V}(\tilde{\delta}_{a})=0$ and $\mathcal{V}(\dot{%
\tilde{\partial}}_{a})=\dot{\tilde{\partial}}_{a}.$ Therefore
\begin{equation*}
\mathcal{V}(\tilde{\partial}_{a})=\left( \rho ,\eta \right) \Gamma _{a}^{b}%
\dot{\tilde{\partial}}_{b}.
\end{equation*}%
Also, it can be deduced that
\begin{equation*}
\begin{array}[b]{c}
\Gamma (\left( \rho ,\eta \right) \Gamma ,Id_{E})(X^{a}\tilde{\partial}%
_{a}+\tilde{X}^{a}\dot{\tilde{\partial}}_{a})=\mathcal{V}(X^{a}\tilde{\partial}%
_{a}+\tilde{X}^{a}\dot{\tilde{\partial}}_{a}),%
\end{array}%
\end{equation*}%
for any
$X^{a}\tilde{\partial}_{a}+\tilde{X}^{a}\dot{\tilde{\partial}}_{a}\in
\Gamma \left( \left( \rho ,\eta \right) TE,\rho \tau _{E},E\right)
.$
\end{example}
\begin{example}
The modules endomorphism
\begin{equation*}
\begin{array}{rcl}
\Gamma \left( \left( \rho ,\eta \right) TE,\left( \rho ,\eta \right) \tau
_{E},E\right) & ^{\underrightarrow{\ \ \mathcal{H}\ \ }} & \Gamma \left(
\left( \rho ,\eta \right) TE,\left( \rho ,\eta \right) \tau _{E},E\right) \\
X^{a}\tilde{\delta}_{a}+\dot{\tilde{X}}^{a}\dot{\tilde{\partial}}_{a} & \longmapsto & X^{a}%
\tilde{\delta}_{a},
\end{array}%
\end{equation*}%
is a projector which is called the \textit{horizontal projector}. It is easy to see that $\mathcal{H}( \tilde{\delta}_{a}) =\tilde{\delta}%
_{\alpha }$ and $\mathcal{H}\big(\dot{\tilde{\partial}}_{a}\big)=0.$
Therefore $\mathcal{H}( \tilde{\partial}_{\alpha })
=\tilde{\delta}_{\alpha }$.
\end{example}
From the above examples we result that any $X\in \Gamma \left(
\left( \rho ,\eta \right) TE,\left( \rho ,\eta
\right) \tau _{E},E\right) $ has the unique decomposition $X=\mathcal{H}X+%
\mathcal{V}X$.
\begin{theorem}
A $\left( \rho ,\eta \right) $-connection for the vector bundle
$(E,\pi ,M)$ is characterized by the existence of a modules
endomorphism $\mathcal{V}$ of $\Gamma \left( \left( \rho ,\eta
\right) TE,\rho \tau _{E},E\right) $ with the properties
\begin{enumerate}
\item[(i)] $\mathcal{V}\left( \Gamma \left( \left( \rho ,\eta \right)
TE,\left( \rho ,\eta \right) \tau _{E},E\right) \right) \subset
\Gamma \left( V\left( \rho ,\eta \right) TE,\left( \rho ,\eta
\right) \tau _{E},E\right) $,

\item[(ii)] $\mathcal{V}\left( X\right) =X\ \ \ if\ \ and \ \  only \ \  if\ \ \ X\in
\Gamma \left( V\left( \rho ,\eta \right) TE,\left( \rho ,\eta
\right) \tau _{E},E\right) $.
\end{enumerate}
\end{theorem}
\begin{theorem}
A $\left( \rho ,\eta \right) $-connection for the vector bundle $\left(
E,\pi ,M\right) $ is characterized by the existence of a modules
endomorphism $\mathcal{H}$ of  $\Gamma \left( \left( \rho ,\eta \right)
TE,\left( \rho ,\eta \right) \tau _{E},E\right) $ with the properties
\begin{enumerate}
\item[(i)] $\mathcal{H}\left( \Gamma \left( \left( \rho ,\eta \right)
TE,\left( \rho ,\eta \right) \tau _{E},E\right) \right) \subset \Gamma
\left( H\left( \rho ,\eta \right) TE,\left( \rho ,\eta \right) \tau
_{E},E\right) $,

\item[(ii)] $\mathcal{H}\left( X\right) =X\ \ \ if \ \ and\ \  only\ \  if\ \ \ X\in
\Gamma \left( H\left( \rho ,\eta \right) TE,\left( \rho ,\eta \right) \tau
_{E},E\right) $.
\end{enumerate}
\end{theorem}
\begin{cor}
A $\left( \rho ,\eta \right) $-connection for the vector bundle $\left(
E,\pi ,M\right) $ is characterized by the existence of a modules
endomorphism $\mathcal{H}$ of  $\Gamma \left( \left( \rho ,\eta \right)
TE,\left( \rho ,\eta \right) \tau _{E},E\right) $ with the properties
\begin{enumerate}
\item[(i)] $\mathcal{H}^{2}=\mathcal{H}$,

\item[(ii)] $Ker\left( \mathcal{H}\right) =\left( \Gamma \left( V\left( \rho
,\eta \right) TE,\left( \rho ,\eta \right) \tau _{E},E\right) ,+,\cdot
\right) $.
\end{enumerate}
\end{cor}
\begin{defn}
A modules endomorphism $e$ of $\Gamma \left( \left( \rho ,\eta
\right)
TE,\break (\left( \rho ,\eta \right) \tau _{E},E\right) $ with the property $%
e^{2}=Id$ is called the almost product structure.
\end{defn}
\begin{example}
The modules endomorphism
\begin{equation}
\begin{array}{rcl}
\Gamma \left( \left( \rho ,\eta \right) TE,\left( \rho ,\eta \right) \tau
_{E},E\right) & ^{\underrightarrow{\ \ \mathcal{P}\ \ }} & \Gamma \left(
\left( \rho ,\eta \right) TE,\left( \rho ,\eta \right) \tau _{E},E\right) \\
X^{a}\tilde{\delta}_{a}+\dot{\tilde{X}}^{a}\dot{\tilde{\partial}}_{a} & \longmapsto & X^{a}%
\tilde{\delta}_{a}-\dot{\tilde{X}}^{a}\dot{\tilde{\partial}}_{a},%
\end{array}
\label{4}
\end{equation}%
is an almost product structure. Thus we have $\mathcal{P}(\tilde{\delta}_{a})=\tilde{\delta}_{a}$ and $\mathcal{P}(%
\dot{\tilde{\partial}}_{a})=-\dot{\tilde{\partial}}_{a}$, which give us $\mathcal{P}(%
\tilde{\partial}_{\alpha })=\tilde{\delta}_{a}-\left( \rho ,\eta
\right) \Gamma _{a}^{b}\dot{\tilde{\partial}}_{b}$.
\end{example}
\begin{theorem}
A $\left( \rho ,\eta \right) $-connection for the vector bundle $\left(
E,\pi ,M\right) $ is characterized by the existence of a modules
endomorphism $\mathcal{P}$ of $\Gamma \left( \left( \rho ,\eta \right)
TE,\left( \rho ,\eta \right) \tau _{E},E\right) $ by the equivalence
\begin{equation*}
\mathcal{P}\left( X\right) =-X\ \ \ if\ \  and\ \  only\ \  if\ \ \ X\in \Gamma \left(
V\left( \rho ,\eta \right) TE,\left( \rho ,\eta \right) \tau _{E},E\right) .
\end{equation*}
\end{theorem}
It is known that the following relations hold between the vertical
projector, horizontal projector and almost product structure.
\begin{equation*}
\mathcal{P}=\left( 2\mathcal{H}-Id\right) ; \ \ \ \ \mathcal{P}=\left( Id-2\mathcal{V}\right) ; \ \ \ \ \mathcal{P}=\left( \mathcal{H}-\mathcal{V}\right) .
\end{equation*}
\begin{defn}
A modules endomorphism $e$ of $\left( \Gamma \!\left( \rho ,\eta
\right)
TE,\break \left( \rho ,\eta \right) \tau _{E},E\right) $ with the property $%
e^{2}=0$ is called the almost tangent structure.
\end{defn}
\begin{example}
If $g$ is a manifolds morphism on $E$ such that $\left(
g,h\right) $ is a locally invertible vector bundles morphism, then the
modules endomorphism
\begin{equation*}
\begin{array}{rcl}
\Gamma \left( \rho TE,\rho \tau _{E},E\right) & ^{\underrightarrow{\mathcal{J%
}_{\left( g,h\right) }}} & \Gamma \left( \rho TE,\rho \tau _{E},E\right) \\
X^{a}\tilde{\delta}_{a}+\dot{\tilde{X}}^{a}\dot{\tilde{\partial}}_{a}
& \longmapsto &
\left( \tilde{g}_{a}^{b}\circ h\circ \pi \right) X^{a}\dot{\tilde{\partial}}%
_{b},%
\end{array}%
\end{equation*}%
is an almost tangent structure which is called the almost tangent
structure associated to the vector bundles morphism $\left( g,h\right) $.
This almost tangent structure has the properties
\begin{equation*}
\mbox{$\mathcal{J}_{( g,h)
}( \tilde{\delta}_{a}) =\mathcal{J}_{( g,h) }(
\tilde{\partial}_{a}) =( \tilde{g}_{a}^{b}\circ h\circ \pi
) \dot{\tilde{\partial}}_{b}$ and $\mathcal{J}_{(
g,h) }( \dot{\tilde{\partial}}_{b}) =0.$}
\end{equation*}%
Moreover, the following equations are hold.
\begin{equation*}
\begin{array}{rcl}
\mathcal{J}_{_{\left( g,h\right) }}\circ \mathcal{P} & = & \mathcal{J}%
_{_{\left( g,h\right) }},\ \ \
\mathcal{P}\circ \mathcal{J}_{_{\left( g,h\right) }}  = -\mathcal{J}%
_{_{\left( g,h\right) }},\ \ \
\mathcal{J}_{_{\left( g,h\right) }}\circ \mathcal{H} =  \mathcal{J}%
_{_{\left( g,h\right) }},\vspace*{1mm} \\
\mathcal{H}\circ \mathcal{J}_{_{\left( g,h\right) }} & = & 0,\ \ \
\mathcal{J}_{_{\left( g,h\right) }}\circ \mathcal{V}= 0,\ \ \
\mathcal{V}\circ \mathcal{J}_{_{\left( g,h\right) }}= \mathcal{J}%
_{_{\left( g,h\right) }}.%
\end{array}%
\end{equation*}
\end{example}
\section{Distinguished linear $\left( \protect\rho ,\protect\eta \right)$%
-connections}
Let $
\left( \mathcal{T}~_{q,s}^{p,r}\left( \left( \rho ,\eta \right) TE,\left(
\rho ,\eta \right) \tau _{E},E\right) ,+,\cdot \right)
$
be the $\mathcal{F}\left( E\right) $-module of tensor fields by $\left(
_{q,s}^{p,r}\right) $-type from the generalized tangent bundle
\begin{equation*}
\left( H\left( \rho ,\eta \right) TE,\rho \tau _{E},E\right) \oplus \left(
\left( \rho ,\eta \right) V\rho TE,\rho \tau _{E},E\right) .
\end{equation*}
An arbitrary tensor field $T$ is written as
\begin{equation*}
\begin{array}{c}
T=T_{b_{1}...b_{q}e_{1}...e_{s}}^{a_{1}...a_{p}d_{1}...d_{r}}\tilde{\delta}%
_{a_{1}}\otimes ...\otimes \tilde{\delta}_{a_{p}}\otimes d\tilde{x}%
^{b_{1}}\otimes ...\otimes d\tilde{x}^{b_{q}}\otimes \\
\dot{\tilde{\partial}}_{d_{1}}\otimes ...\otimes \dot{\tilde{\partial}}%
_{d_{r}}\otimes \delta \tilde{y}^{e_{1}}\otimes ...\otimes \delta \tilde{y}%
^{e_{s}}.%
\end{array}%
\end{equation*}
Let
$
\left( \mathcal{T}~\left( \left( \rho ,\eta \right) TE,\left( \rho ,\eta
\right) \tau _{E},E\right) ,+,\cdot ,\otimes \right)
$
be the tensor fields algebra of generalized tangent bundle $( (
\rho ,\eta ) \rho TE,( \rho ,\eta ) \rho \tau _{E},E)
$. Moreover, let $\left( E,\pi ,M\right) $ be a vector bundle endowed with a $\left( \rho
,\eta \right) $-connection $\left( \rho ,\eta \right) \Gamma $ and
\begin{equation*}
\begin{array}{l}
\left( X,T\right) ^{\ \underrightarrow{\left( \rho ,\eta \right) D}\,}%
\vspace*{1mm}\left( \rho ,\eta \right) D_{X}T,
\end{array}%
\end{equation*}%
be a covariant $\left( \rho ,\eta \right) $-derivative for the tensor
algebra
\begin{equation*}
\left( \mathcal{T}~\left( \left( \rho ,\eta \right) TE,\left( \rho ,\eta
\right) \tau _{E},E\right) ,+,\cdot ,\otimes \right),
\end{equation*}%
of the generalized tangent bundle
$
\left( \left( \rho ,\eta \right) TE,\left( \rho ,\eta \right) \tau
_{E},E\right),
$
which preserves the horizontal and vertical interior differential systems by
parallelism.
The real local functions
\begin{equation*}
\left( \left( \rho ,\eta \right) H_{bc}^{a},\left( \rho ,\eta \right) \tilde{%
H}_{bc}^{a},\left( \rho ,\eta \right) V_{bc}^{a},\left( \rho ,\eta \right)
\tilde{V}_{bc}^{a}\right),
\end{equation*}%
defined by the equalities
\begin{equation*}
\begin{array}{ll}
\left( \rho ,\eta \right) D_{\tilde{\delta}_{c}}\tilde{\delta}_{b}=\left(
\rho ,\eta \right) H_{bc}^{a}\tilde{\delta}_{a}, & \left( \rho ,\eta \right)
D_{\tilde{\delta}_{c}}\dot{\tilde{\partial}}_{b}=\left( \rho ,\eta \right)
\tilde{H}_{bc}^{a}\dot{\tilde{\partial}}_{a} \\
\left( \rho ,\eta \right) D_{\dot{\tilde{\partial}}_{c}}\tilde{\delta}%
_{b}=\left( \rho ,\eta \right) V_{bc}^{a}\tilde{\delta}_{a}, & \left( \rho
,\eta \right) D_{\dot{\tilde{\partial}}_{c}}\dot{\tilde{\partial}}%
_{b}=\left( \rho ,\eta \right) \tilde{V}_{bc}^{a}\dot{\tilde{\partial}}_{a},
\end{array}%
\end{equation*}%
are the components of a linear $\left( \rho ,\eta \right) $-connection $%
\left( \left( \rho ,\eta \right) H,\left( \rho ,\eta \right) V\right) $ for
the generalized tangent bundle $\left( \left( \rho ,\eta \right) TE,\left(
\rho ,\eta \right) \tau _{E},E\right) $ which will be called the
{\it distinguished linear $\left( \rho ,\eta \right) $-connection}.
\begin{remark}
The distinguished linear $\left( Id_{TM},Id_{M}\right) $-connection is the
classical distinguished linear connection. The components of a
distinguished linear connection $\left( H,V\right) $ will be denoted by $%
(
H_{jk}^{i},\tilde{H}_{jk}^{i},V_{jk}^{i},\tilde{V}_{jk}^{i})$.
\end{remark}
\begin{theorem}
If the generalized tangent bundle $\!(\!\left( \rho ,\eta \right)
T\!E,\!\left( \rho ,\eta \right) \tau _{E},\!E\!)$ is endowed with a
distinguished linear $\!\left( \rho ,\eta \right) $-connection $(\left( \rho
,\eta \right) H,\left( \rho ,\eta \right) V),$ then for any
\begin{equation*}
\begin{array}[b]{c}
X=X^{a}\tilde{\delta}_{a}+\dot{\tilde{X}}^{a}\dot{\tilde{\partial}}_{a}\in
\Gamma (\!\left( \rho ,\eta \right) TE,\!\left( \rho ,\eta \right)
\tau _{E},\!E),
\end{array}%
\end{equation*}%
and for any
\begin{equation*}
T\in \mathcal{T}_{qs}^{pr}\!(\!\left( \rho ,\eta \right) TE,\!\left( \rho
,\eta \right) \tau _{E},\!E),
\end{equation*}%
we have
\begin{equation}\label{good}
\begin{array}{l}
\left( \rho ,\eta \right) D_{X}\left(
T_{b_{1}...b_{q}e_{1}\cdots e_{s}}^{a_{1}\cdots a_{p}d_{1}\cdots d_{r}}\tilde{\delta}%
_{a_{1}}\otimes\cdots \otimes \tilde{\delta}_{a_{p}}\otimes d\tilde{x}%
^{b_{1}}\otimes\cdots \otimes \right. \vspace*{1mm} \\
\hspace*{9mm}\left. \otimes d\tilde{x}^{b_{q}}\otimes \dot{\tilde{\partial}}%
_{d_{1}}\otimes\cdots\otimes \dot{\tilde{\partial}}_{d_{r}}\otimes
\delta
\tilde{y}^{e_{1}}\otimes\cdots\otimes \delta \tilde{y}^{e_{s}}\right) =%
\vspace*{1mm} \\
\hspace*{9mm}=X^{c}T_{b_{1}\cdots b_{q}e_{1}\cdots e_{s}\mid
c}^{a_{1}\cdots a_{p}d_{1}\cdots d_{r}}\tilde{\delta}_{a_{1}}\otimes
\cdots\otimes
\tilde{\delta}_{a_{p}}\otimes d\tilde{x}^{b_{1}}\otimes\cdots\otimes d\tilde{x}%
^{b_{q}}\otimes \dot{\tilde{\partial}}_{d_{1}}\otimes\cdots\otimes \vspace*{1mm%
} \\
\hspace*{9mm}\otimes \dot{\tilde{\partial}}_{d_{r}}\otimes \delta \tilde{y}%
^{e_{1}}\otimes\cdots\otimes \delta \tilde{y}%
^{e_{s}}+\dot{\tilde{X}}^{c}T_{b_{1}\cdots b_{q}e_{1}\cdots e_{s}}^{a_{1}\cdots a_{p}d_{1}\cdots d_{r}}%
\mid _{c}\tilde{\delta}_{a_{1}}\otimes\cdots\otimes \vspace*{1mm}\\
\hspace*{9mm}\otimes \tilde{\delta}_{a_{p}}\otimes
d\tilde{x}^{b_{1}}\otimes\cdots\otimes d\tilde{x}^{b_{q}}\otimes
\dot{\tilde{\partial}}_{d_{1}}\otimes
\cdots\otimes \dot{\tilde{\partial}}_{d_{r}}\otimes \delta \tilde{y}%
^{e_{1}}\otimes\cdots\otimes \delta \tilde{y}^{e_{s}},%
\end{array}%
\end{equation}%
where
\begin{equation*}
\begin{array}{l}
T_{b_{1}\cdots b_{q}e_{1}\cdots e_{s}\mid c}^{a_{1}\cdots a_{p}d_{1}\cdots d_{r}}=\vspace*{%
2mm}\Gamma \left( \tilde{\rho},Id_{E}\right) ( \tilde{\delta}%
_{c}) T_{b_{1}\cdots b_{q}e_{1}\cdots e_{s}}^{a_{1}\cdots a_{p}d_{1}\cdots d_{r}} \\
\hspace*{8mm}+\left( \rho ,\eta \right)
H_{ac}^{a_{1}}T_{b_{1}\cdots b_{q}e_{1}\cdots e_{s}}^{aa_{2}\cdots a_{p}d_{1}\cdots d_{r}}+\cdots+%
\vspace*{2mm}\left( \rho ,\eta \right) H_{ac}^{a_{p}}T_{b_{1}\cdots
b_{q}e_{1}\cdots e_{s}}^{a_{1}\cdots a_{p-1}ad_{1}\cdots d_{r}}
\\
\hspace*{8mm}-\left( \rho ,\eta \right)
H_{b_{1}c}^{b}T_{bb_{2}\cdots b_{q}e_{1}\cdots e_{s}}^{a_{1}\cdots a_{p}d_{1}\cdots d_{r}}-\cdots -%
\vspace*{2mm}\left( \rho ,\eta \right) H_{b_{q}c}^{b}T_{b_{1}\cdots
b_{q-1}be_{1}\cdots e_{s}}^{a_{1}\cdots a_{p}d_{1}\cdots d_{r}}
\\
\hspace*{8mm}+\left( \rho ,\eta \right) \tilde{H}%
_{dc}^{d_{1}}T_{b_{1}\cdots b_{q}e_{1}\cdots e_{s}}^{a_{1}\cdots a_{p}dd_{2}\cdots d_{r}}+\cdots +%
\vspace*{2mm}\left( \rho ,\eta \right) \tilde{H}%
_{dc}^{d_{r}}T_{b_{1}\cdots b_{q}e_{1}\cdots e_{s}}^{a_{1}\cdots
a_{p}d_{1}\cdots
d_{r-1}d}
\\
\hspace*{8mm}-\left( \rho ,\eta \right) \tilde{H}%
_{e_{1}c}^{e}T_{b_{1}\cdots b_{q}ee_{2}ve_{s}}^{a_{1}\cdots a_{p}d_{1}\cdots d_{r}}-%
\vspace*{2mm}\cdots -\left( \rho ,\eta \right) \tilde{H}%
_{e_{s}c}^{e}T_{b_{1}\cdots b_{q}e_{1}\cdots e_{s-1}e}^{a_{1}\cdots
a_{p}d_{1}\cdots d_{r}},
\end{array}%
\end{equation*}%
and
\begin{equation*}
\begin{array}{l}
T_{b_{1}\cdots b_{q}e_{1}\cdots e_{s}}^{a_{1}\cdots a_{p}d_{1}\cdots d_{r}}\mid _{c}=%
\vspace*{2mm}\Gamma \left( \tilde{\rho},Id_{E}\right) ( \dot{\tilde{%
\partial}}_{c})
T_{b_{1}\cdots b_{q}e_{1}\cdots e_{s}}^{a_{1}\cdots a_{p}d_{1}\cdots d_{r}} \\
\hspace*{8mm}+\left( \rho ,\eta \right) V_{ac}^{a_{1}}T_{b_{1}\cdots
b_{q}e_{1}\cdots e_{s}}^{aa_{2}\cdots a_{p}d_{1}\cdots d_{r}}+\cdots
+\left( \rho ,\eta \right) \rho V_{ac}^{a_{p}}T_{b_{1}\cdots
b_{q}e_{1}\cdots e_{s}}^{a_{1}\cdots a_{p-1}ad_{1}\cdots d_{r}}
\\
\hspace*{8mm}-\left( \rho ,\eta \right)
V_{b_{1}c}^{b}T_{bb_{2}\cdots b_{q}e_{1}\cdots e_{s}}^{a_{1}\cdots a_{p}d_{1}\cdots d_{r}}-\cdots -%
\vspace*{2mm}\left( \rho ,\eta \right) V_{b_{q}c}^{b}T_{b_{1}\cdots
b_{q-1}be_{1}\cdots e_{s}}^{a_{1}\cdots a_{p}d_{1}\cdots d_{r}}
\\
\hspace*{8mm}+\left( \rho ,\eta \right) \tilde{V}%
_{dc}^{d_{1}}T_{b_{1}\cdots b_{q}e_{1}\cdots e_{s}}^{a_{1}\cdots a_{p}dd_{2}\cdots d_{r}}+\cdots +%
\vspace*{2mm}\left( \rho ,\eta \right) \tilde{V}%
_{dc}^{d_{r}}T_{b_{1}\cdots b_{q}e_{1}\cdots e_{s}}^{a_{1}\cdots
a_{p}d_{1}\cdots d_{r-1}d}
\\
\hspace*{8mm}-\left( \rho ,\eta \right) \tilde{V}%
_{e_{1}c}^{e}T_{b_{1}\cdots b_{q}ee_{2}\cdots e_{s}}^{a_{1}\cdots a_{p}d_{1}\cdots d_{r}}-%
\vspace*{2mm}\cdots -\left( \rho ,\eta \right) \tilde{V}%
_{e_{s}c}^{e}T_{b_{1}\cdots b_{q}e_{1}\cdots e_{s-1}e}^{a_{1}\cdots
a_{p}d_{1}\cdots d_{r}}.
\end{array}%
\end{equation*}
\end{theorem}
\begin{defn}
If $\left( \rho ,\eta \right) \Gamma $ is a $\left( \rho ,\eta \right) $%
-connection for the vector bundle $\left( E,\pi ,M\right) $ and
\begin{equation*}
( \left( \rho ,\eta \right) H_{bc}^{a},\left( \rho ,\eta \right) \tilde{%
H}_{bc}^{a},\left( \rho ,\eta \right) V_{bc}^{a},\left( \rho ,\eta \right)
\tilde{V}_{bc}^{a}),
\end{equation*}%
are the components of a distinguished linear $\left( \rho ,\eta \right) $%
\textit{-}connection for the generalized tangent bundle $\left( \left( \rho
,\eta \right) TE,\left( \rho ,\eta \right) \tau _{E},E\right) $ such that
\begin{equation*}
\left( \rho ,\eta \right) H_{bc}^{a}=\left( \rho ,\eta \right) \tilde{H}%
_{bc}^{a}\mbox{ and }\left( \rho ,\eta \right) V_{bc}^{a}=\left( \rho ,\eta
\right) \tilde{V}_{bc}^{a},
\end{equation*}%
then we will say that the generalized tangent bundle $\!(\!\left( \rho ,\eta
\right) TE,\left( \rho ,\eta \right) \tau _{E},\!E)$ is endowed with a
normal distinguished linear $\left( \rho ,\eta \right) $-connection $\left(
\left( \rho ,\eta \right) H,\left( \rho ,\eta \right) V\right) $ with
components $\left( \left( \rho ,\eta \right) H_{bc}^{a},\left( \rho ,\eta
\right) V_{bc}^{a}\right) $.
\end{defn}
\begin{example}
The local real functions%
\begin{equation*}
\begin{array}[b]{c}
( \frac{\partial \left( \rho ,\eta \right) \Gamma
_{c}^{a}}{\partial y^{b}},\frac{\partial \left( \rho ,\eta \right)
\Gamma _{c}^{a}}{\partial
y^{b}},0,0),
\end{array}%
\end{equation*}%
are the components of a normal distinguished linear $\rho
$\textit{-}connection for the generalized tangent bundle $\left(
\left( \rho ,\eta \right) TE,\left(
\rho ,\eta \right) \tau _{E},E\right) ,$ which will be called the \emph{%
Berwald linear }$\left( \rho ,\eta \right) $\emph{-connection.}
\end{example}

It is remarkable that the Berwald linear $\left( Id_{TM},Id_{M}\right) $-connection is the
usual \emph{Ber\-wald linear connection.}
\subsection{The $(\protect\rho ,\protect\eta )$-torsion and the $(\protect%
\rho ,\protect\eta )$-curvature of a distinguished linear $(\protect\rho ,%
\protect\eta )$-connection}
We consider the generalized Lie algebroid $(\left( E,\pi ,M\right) ,\left[ ,\right] _{E,h},\left( \rho
,\eta \right))$ and we let $\left( \rho ,\eta \right) \Gamma $ be a $\left( \rho ,\eta \right) $%
-connection for the vector bundle $\left( E,\pi ,M\right) $ and $\left(
\left( \rho ,\eta \right) H,\left( \rho ,\eta \right) V\right) $ be a
distinguished linear $\left( \rho ,\eta \right) $-connection for the
generalized tangent bundle $( ( \rho ,\eta ) TE$ $,( \rho
,\eta ) \tau _{E},E)$.
\begin{defn}
Application
\begin{equation*}
\begin{array}{rcl}
\Gamma \left( \left( \rho ,\eta \right) TE,\left( \rho ,\eta \right) \tau
_{E},E\right) ^{2} & ^{\underrightarrow{\left( \rho ,\eta ,h\right) \mathbb{T%
}}} & \Gamma \left( \left( \rho ,\eta \right) TE,\left( \rho ,\eta \right)
\tau _{E},E\right) \vspace*{2mm} \\
\left( X,Y\right) & \longmapsto & \left( \rho ,\eta ,h\right) \mathbb{T}%
\left( X,Y\right),%
\end{array}%
\end{equation*}%
defined by
\begin{equation}
\begin{array}{c}
\left( \rho ,\eta ,h\right) \mathbb{T}\left( X,Y\right) =\left( \rho ,\eta
\right) D_{X}Y-\left( \rho ,\eta \right) D_{Y}X-\left[ X,Y\right] _{\left(
\rho ,\eta \right) TE},%
\end{array}%
\end{equation}%
for any $X,Y\in \Gamma \left( \left( \rho ,\eta \right) TE,\left( \rho ,\eta
\right) \tau _{E},E\right) ,$ is called the $\left( \rho ,\eta ,h\right) $%
\textit{-torsion associated to the distinguished linear }$\left( \rho ,\eta
\right) $\textit{-connection }$\left( \left( \rho ,\eta \right) H,\left(
\rho ,\eta \right) V\right) .$

Applications
\begin{equation*}
\mathcal{H}\left( \rho ,\eta ,h\right) \mathbb{T}\left( \mathcal{H}\left(
\cdot \right) ,\mathcal{H}\left( \cdot \right) \right) ,\,\mathcal{V}\left(
\rho ,\eta ,h\right) \mathbb{T}\left( \mathcal{H}\left( \cdot \right) ,%
\mathcal{H}\left( \cdot \right) \right) ,...,\mathcal{V}\left( \rho ,\eta
,h\right) \mathbb{T}\left( \mathcal{V}\left( \cdot \right) ,\mathcal{V}%
\left( \cdot \right) \right),
\end{equation*}%
are called $\mathcal{H}\left( \mathcal{HH}\right) ,\,\mathcal{V}\left(
\mathcal{HH}\right) ,...,\mathcal{V}\left( \mathcal{VV}\right) $ $\left(
\rho ,\eta ,h\right) $-torsions associated to the distinguished linear $\left(
\rho ,\eta \right) $-connection $\left( \left( \rho ,\eta \right) H,\left(
\rho ,\eta \right) V\right) .$
\end{defn}
Using the notations
\begin{equation}
\begin{array}{l}
\mathcal{H}\left( \rho ,\eta ,h\right) \mathbb{T}( \tilde{\delta}_{c},%
\tilde{\delta}_{b}) =\left( \rho ,\eta ,h\right) \mathbb{T}_{~bc}^{a}%
\tilde{\delta}_{a}, \\
\mathcal{V}\left( \rho ,\eta ,h\right) \mathbb{T}( \tilde{\delta}%
_{\gamma },\tilde{\delta}_{\beta }) =\left( \rho ,\eta ,h\right)
\mathbb{\tilde{T}}_{~bc}^{a}\dot{\tilde{\partial}}_{a}, \\
\mathcal{H}\left( \rho ,\eta ,h\right) \mathbb{T}( \dot{\tilde{\partial}%
}_{c},\tilde{\delta}_{b}) =\left( \rho ,\eta ,h\right) \mathbb{P}%
_{~bc}^{a}\tilde{\delta}_{a}, \\
\mathcal{V}\left( \rho ,\eta ,h\right) \mathbb{T}( \dot{\tilde{\partial}%
}_{c},\tilde{\delta}_{b}) =\left( \rho ,\eta ,h\right) \mathbb{\tilde{P%
}}_{~bc}^{a}\dot{\tilde{\partial}}_{a}, \\
\ \mathcal{V}\left( \rho ,\eta ,h\right) \mathbb{T}( \dot{\tilde{%
\partial}}_{c},\dot{\tilde{\partial}}_{b}) =\left( \rho ,\eta
,h\right) \mathbb{S}_{~bc}^{a}\dot{\tilde{\partial}}_{a},%
\end{array}%
\end{equation}%
one can deduce that the $\left( \rho ,\eta ,h\right) $-torsion
$\left( \rho ,\eta ,h\right) \mathbb{T}$ associated to the
distinguished linear $\left( \rho ,\eta \right) $ -connection
$\left( \left( \rho ,\eta \right) H,\left( \rho ,\eta \right)
V\right) $ is characterized by the tensor fields with the local
components
\begin{equation}
\begin{array}{cl}
\left( \rho ,\eta ,h\right) \mathbb{T}_{~bc}^{a} & =\left( \rho ,\eta
\right) H_{bc}^{a}-\left( \rho ,\eta \right) H_{cb}^{a}-L_{bc}^{a}\circ
h\circ \pi , \\
\left( \rho ,\eta ,h\right) \mathbb{\tilde{T}}_{~bc}^{a} & =\left( \rho
,\eta ,h\right) \mathbb{R}_{\,\ bc}^{a}, \\
\left( \rho ,\eta ,h\right) \mathbb{P}_{~bc}^{a} & =\left( \rho ,\eta
\right) V_{bc}^{a}, \\
\left( \rho ,\eta ,h\right) \mathbb{\tilde{P}}_{~\beta c}^{a} & =%
\displaystyle\frac{\partial }{\partial y^{c}}\left( \left( \rho ,\eta
\right) \Gamma _{b}^{a}\right) -\left( \rho ,\eta \right) \tilde{H}_{cb}^{a},
\\
\left( \rho ,\eta ,h\right) \mathbb{S}_{~bc}^{a} & =\left( \rho ,\eta
\right) \tilde{V}_{bc}^{a}-\left( \rho ,\eta \right) \tilde{V}_{cb}^{a}.%
\end{array}%
\end{equation}
\begin{defn}
The application
\begin{equation*}
\begin{array}{rcc}
\left( \Gamma \left( \left( \rho ,\eta \right) TE,\left( \rho ,\eta \right)
\tau _{E},E\right) \right) ^{3} & ^{\underrightarrow{\ \left( \rho ,\eta
,h\right) \mathbb{R}\ }} & \Gamma \left( \left( \rho ,\eta \right) TE,\left(
\rho ,\eta \right) \tau _{E},E\right) \vspace*{3mm} \\
\left( \left( Y,Z\right) ,X\right) & \longmapsto & \left( \rho ,\eta
,h\right) \mathbb{R}\left( Y,Z\right)X,%
\end{array}%
\end{equation*}%
by the rule
\begin{equation}
\begin{array}{l}
\left( \rho ,\eta ,h\right) \mathbb{R}\left( Y,Z\right) X=\left( \rho ,\eta
\right) D_{Y}\left( \left( \rho ,\eta \right) D_{Z}X\right) \vspace*{2mm} \\
\qquad -\left( \rho ,\eta \right) D_{Z}\left( \left( \rho ,\eta \right)
D_{Y}X\right) -\left( \rho ,\eta \right) D_{\left[ Y,Z\right] _{\left( \rho
,\eta \right) TE}}X,\,%
\end{array}%
\end{equation}%
for any $X,Y,Z\in \Gamma \left( \left( \rho ,\eta \right) TE,\left( \rho
,\eta \right) \tau _{E},E\right)$, is called the $\left( \rho ,\eta
,h\right) $-curvature associated to the distinguished linear $\left(
\rho ,\eta \right) $\emph{-connection }$\left( \left( \rho ,\eta \right)
H,\left( \rho ,\eta \right) V\right) .$
\end{defn}
Using the notations
\begin{equation}
\begin{array}{rl}
\left( \rho ,\eta ,h\right) \mathbb{R}( \tilde{\delta}_{d},\tilde{\delta%
}_{c}) \tilde{\delta}_{b} & =\left( \rho ,\eta ,h\right) \mathbb{R}%
_{~b~cd}^{a}\tilde{\delta}_{a}, \\
\left( \rho ,\eta ,h\right) \mathbb{R}( \tilde{\delta}_{d},\tilde{\delta%
}_{c}) \dot{\tilde{\partial}}_{b} & =\left( \rho ,\eta ,h\right)
\mathbb{\tilde{R}}_{~b~cd}^{a}\dot{\tilde{\partial}}_{a}, \\
\left( \rho ,\eta ,h\right) \mathbb{R}( \dot{\tilde{\partial}}_{d},%
\tilde{\delta}_{c}) \tilde{\delta}_{b} & =\left( \rho ,\eta ,h\right)
\mathbb{P}_{~b~cd}^{a}\tilde{\delta}_{a}, \\
\left( \rho ,\eta ,h\right) \mathbb{R}( \dot{\tilde{\partial}}_{d},%
\tilde{\delta}_{c}) \dot{\tilde{\partial}}_{b} & =\left( \rho ,\eta
,h\right) \mathbb{\tilde{P}}_{~b~cd}^{a}\dot{\tilde{\partial}}_{a}, \\
\left( \rho ,\eta ,h\right) \mathbb{R}( \dot{\tilde{\partial}}_{d},\dot{%
\tilde{\partial}}_{c}) \tilde{\delta}_{b} & =\left( \rho ,\eta
,h\right) \mathbb{S}_{~b~cd}^{a}\tilde{\delta}_{a}, \\
\left( \rho ,\eta ,h\right) \mathbb{R}( \dot{\tilde{\partial}}_{d},\dot{%
\tilde{\partial}}_{c}) \dot{\tilde{\partial}}_{b} & =\left( \rho ,\eta
,h\right) \mathbb{\tilde{S}}_{~b~cd}^{a}\dot{\tilde{\partial}}_{a},%
\end{array}%
\end{equation}%
the following will be yield easily.

\begin{theorem}
The $\left( \rho ,\eta ,h\right) $-curvature $\left( \rho ,\eta ,h\right)
\mathbb{R}$ associated to the distinguished linear $\left( \rho ,\eta \right) $
-connection $\left( \left( \rho ,\eta \right) H,\left( \rho ,\eta \right)
V\right) $ is characterized by the\ tensor fields with local components
\begin{equation}  \label{reclaim1}
\left\{
\begin{array}{cl}
(\rho ,\eta ,h)\mathbb{R}_{~b~cd}^{a} & {=}\Gamma (\tilde{\rho}%
,Id_{E})\!( \!\tilde{\delta}_{d}\!) \!(\rho ,\eta )H_{bc}^{a}{-}%
\Gamma \!(\tilde{\rho},Id_{E})\!( \!\tilde{\delta}_{c}\!) (\rho
,\eta )H_{bd}^{a}\vspace*{2mm} \\
& +(\rho ,\eta )H_{ed}^{a}(\rho ,\eta )H_{bc}^{e}{-}(\rho ,\eta
)H_{ec}^{a}(\rho ,\eta )H_{bd}^{e}\vspace*{2mm} \\
& -(\rho ,\eta ,h)\mathbb{R}_{\,\ cd}^{e}(\rho ,\eta )H_{be}^{a}{-}%
L_{cd}^{e}\circ h\circ \pi (\rho ,\eta )H_{\beta e}^{a},\vspace*{3mm} \\
(\rho ,\eta ,h)\mathbb{\tilde{R}}_{~b~cd}^{a} & {=}\Gamma (\tilde{\rho}%
,Id_{E})\!( \!\tilde{\delta}_{d}\!) \!(\rho ,\eta )\tilde{H}%
_{bc}^{a}{-}\Gamma (\tilde{\rho},Id_{E})\!( \!\tilde{\delta}%
_{c}\!) (\rho ,\eta )\tilde{H}_{bd}^{a}\vspace*{2mm} \\
& +(\rho ,\eta )\tilde{H}_{ed}^{a}(\rho ,\eta )\tilde{H}_{bc}^{e}{-}(\rho
,\eta )\tilde{H}_{ec}^{a}(\rho ,\eta )\tilde{H}_{bd}^{e}\vspace*{2mm} \\
& -(\rho ,\eta ,h)\mathbb{R}_{\,\ dc}^{e}(\rho ,\eta )\tilde{V}_{be}^{a}{-}%
L_{cd}^{e}\circ h\circ \pi (\rho ,\eta )\tilde{V}_{be}^{a},%
\end{array}%
\right.
\end{equation}
\begin{equation}  \label{reclaim2}
\left\{
\begin{array}{ll}
(\rho ,\eta ,h)\mathbb{P}_{~b~cd}^{a}\!\!\! & =\Gamma (\tilde{\rho},Id_{E})(%
\dot{\tilde{\partial}}_{d})(\rho ,\eta )H_{bc}^{a}{-}\vspace*{1mm}\Gamma (%
\tilde{\rho},Id_{E})\!( \!\tilde{\delta}_{c}\!) \!(\rho ,\eta
)V_{bd}^{a}\vspace*{2mm} \\
& {=}(\rho ,\eta )V_{ed}^{a}(\rho ,\eta )H_{bc}^{e}-\vspace*{1mm}(\rho ,\eta
)H_{ec}^{a}(\rho ,\eta )V_{bd}^{e}\vspace*{2mm} \\
& +\displaystyle\frac{\partial }{\partial y^{c}}\left( \left( \rho ,\eta
\right) \Gamma _{c}^{e}\right) \left( \rho ,\eta \right) V_{be}^{a},\vspace*{%
2mm} \\
(\rho ,\eta ,h)\mathbb{\tilde{P}}_{~b~cd}^{a}\!\!\! & =\Gamma (\tilde{\rho}%
,Id_{E})\!( \!\dot{\tilde{\partial}}_{d}\!) \!(\rho ,\eta )\tilde{%
H}_{bc}^{a}-\vspace*{2mm} \\
& -\Gamma (\tilde{\rho},Id_{E})\!( \!\tilde{\delta}_{c}\!)
\!(\rho ,\eta )\tilde{V}_{bd}^{a}+(\rho ,\eta )\tilde{V}_{ed}^{a}(\rho ,\eta
)\tilde{H}_{bc}^{e}-\vspace*{2mm} \\
& -(\rho ,\eta )\tilde{H}_{ec}^{a}(\rho ,\eta )\tilde{V}_{bd}^{e}+%
\displaystyle\frac{\partial }{\partial y^{d}}(\left( \rho ,\eta )\Gamma
_{c}^{e}\right) (\rho ,\eta )\tilde{V}_{be}^{a},%
\end{array}%
\right.
\end{equation}%
\begin{equation}  \label{reclaim3}
\left\{
\begin{array}{cl}
\left( \rho ,\eta ,h\right) \mathbb{S}_{~b~cd}^{a} & =\Gamma \left( \tilde{%
\rho},Id_{E}\right) ( \dot{\tilde{\partial}}_{d}) \left( \rho
,\eta \right) V_{bc}^{a}\vspace*{2mm} \\
& -\Gamma \left( \tilde{\rho},Id_{E}\right) ( \dot{\tilde{\partial}}%
_{c}) \left( \rho ,\eta \right) V_{bd}^{a}+\left( \rho ,\eta \right)
V_{ed}^{a}\left( \rho ,\eta \right) V_{bc}^{e}\vspace*{2mm} \\
& -\left( \rho ,\eta \right) V_{ec}^{e}\left( \rho ,\eta \right) V_{bd}^{e},%
\vspace*{3mm} \\
\left( \rho ,\eta ,h\right) \mathbb{\tilde{S}}_{~b~cd}^{a} & =\Gamma \left(
\tilde{\rho},Id_{E}\right)( \dot{\tilde{\partial}}_{d}) \left(
\rho ,\eta \right) \tilde{V}_{bc}^{a}\vspace*{2mm} \\
& -\Gamma \left( \tilde{\rho},Id_{E}\right) ( \dot{\tilde{\partial}}%
_{c}) \left( \rho ,\eta \right) \tilde{V}_{bd}^{a}+\left( \rho ,\eta
\right) \tilde{V}_{ed}^{a}\left( \rho ,\eta \right) \tilde{V}_{bc}^{e}%
\vspace*{2mm} \\
& -\left( \rho ,\eta \right) \tilde{V}_{ec}^{a}\left( \rho ,\eta \right)
\tilde{V}_{bd}^{e}.%
\end{array}%
\right.
\end{equation}%
\end{theorem}
\begin{defn}
Using the mixed components of $\left( \rho ,\eta ,h\right)
$-curvature, we define the mixed curvature
\[
\begin{array}{ccc}
\Gamma \left( \left( \rho ,\eta \right) TE,\left( \rho ,\eta \right)
\tau _{E},E\right) ^{3} & \underrightarrow{~\ \ \ \mathbb{P~\ \ }} &
\left( \left( \rho ,\eta \right) TE,\left( \rho ,\eta \right) \tau
_{E},E\right)
\\
\left( X,Y,Z\right)  & \longmapsto  & \mathbb{P}\left( X,Y\right)
Z:=\mathbb{R}\left( \mathcal{V}X,\mathcal{H}Y\right) Z.
\end{array}%
\]
\end{defn}
Note that if $X=X^{a}\tilde{\delta}_{a}+\dot{\tilde{X}}^{a}\dot{%
\tilde{\partial}}_{a},$
$Y=Y^{a}\tilde{\delta}_{a}+\dot{\tilde{Y}}^{a}\dot{\tilde{\partial}}_{a}$
and
$Z=Z^{a}\tilde{\delta}_{a}+\dot{\tilde{Z}}^{a}\dot{\tilde{\partial}}_{a},$
then
\[
\begin{array}{ccc}
\mathbb{P}\left( X,Y\right) Z & = &
\dot{\tilde{X}}^{d}Y^{c}Z^{a}\left( \rho ,\eta
,h\right) \mathbb{P}_{~b~cd}^{a}\tilde{\delta}_{a}+\dot{\tilde{X}}^{d}Y^{c}\dot{\tilde{Z}}%
^{a}\left( \rho ,\eta ,h\right)
\mathbb{\tilde{P}}_{~b~cd}^{a}\dot{\tilde{\partial}}_{a}.
\end{array}%
\]
\subsection{Formulas of Ricci type, identities of Cartan and Bianchi type}
We consider the generalized Lie algebroid $(\left( E,\pi ,M\right)
,\left[ ,\right] _{E,h},\left( \rho ,\eta \right))$.
\begin{theorem}
\label{fr9} Let $\left( \rho ,\eta \right) \Gamma $ be a $\left( \rho ,\eta
\right) $-connection for the vector bundle $\left( E,\pi ,M\right) $ and $%
\left( \left( \rho ,\eta \right) H,\left( \rho ,\eta \right) V\right) $ a
distinguished linear $\left( \rho ,\eta \right) $-connection for the
generalized tangent bundle $\left( \left( \rho ,\eta \right) TE,\left( \rho
,\eta \right) \tau _{E},E\right)$. Then
\begin{equation}
\left\{
\begin{array}{l}
\begin{array}{l}
\left( \rho ,\eta \right) D_{\mathcal{H}X}\left( \rho ,\eta \right) D_{%
\mathcal{H}Y}\mathcal{H}Z-\left( \rho ,\eta \right) D_{\mathcal{H}Y}\left(
\rho ,\eta \right) D_{\mathcal{H}X}\mathcal{H}Z \\
=\left( \rho ,\eta ,h\right) \mathbb{R}\left( \mathcal{H}X,\mathcal{H}%
Y\right) \mathcal{H}Z+\left( \rho ,\eta \right) D_{\mathcal{H}\left[
\mathcal{H}X,\mathcal{H}Y\right] _{\left( \rho ,\eta \right) TE}}\mathcal{H}Z
\\
+\left( \rho ,\eta \right) D_{\mathcal{V}\left[ \mathcal{H}X,\mathcal{H}Y%
\right] _{\left( \rho ,\eta \right) TE}}\mathcal{H}Z,%
\end{array}
\\
\multicolumn{1}{c}{%
\begin{array}{l}
\left( \rho ,\eta \right) D_{\mathcal{V}X}\left( \rho ,\eta \right) D_{%
\mathcal{H}Y}\mathcal{H}Z-\left( \rho ,\eta \right) D_{\mathcal{H}Y}\left(
\rho ,\eta \right) D_{\mathcal{V}X}\mathcal{H}Z \\
=\left( \rho ,\eta ,h\right) \mathbb{R}\left( \mathcal{V}X,\mathcal{H}%
Y\right) \mathcal{H}Z+\left( \rho ,\eta \right) D_{\mathcal{H}\left[
\mathcal{V}X,\mathcal{H}Y\right] _{\left( \rho ,\eta \right) TE}}\mathcal{H}Z
\\
+\left( \rho ,\eta \right) D_{\mathcal{V}\left[ \mathcal{V}X,\mathcal{H}Y%
\right] _{\left( \rho ,\eta \right) TE}}\mathcal{H}Z,%
\end{array}%
} \\
\begin{array}{c}
\left( \rho ,\eta \right) D_{\mathcal{V}X}\left( \rho ,\eta \right) D_{%
\mathcal{V}Y}\mathcal{H}Z-\left( \rho ,\eta \right) D_{\mathcal{V}Y}\left(
\rho ,\eta \right) D_{\mathcal{V}X}\mathcal{H}Z \\
=\left( \rho ,\eta ,h\right) \mathbb{R}\left( \mathcal{V}X,\mathcal{V}%
Y\right) \mathcal{H}Z+\left( \rho ,\eta \right) D_{\mathcal{V}\left[
\mathcal{V}X,\mathcal{V}Y\right] _{\left( \rho ,\eta \right) TE}}\mathcal{H}%
Z,%
\end{array}%
\end{array}%
\right.
\end{equation}
and
\begin{equation}
\left\{
\begin{array}{c}
\begin{array}{l}
\left( \rho ,\eta \right) D_{\mathcal{H}X}\left( \rho ,\eta \right) D_{%
\mathcal{H}Y}\mathcal{V}Z-\left( \rho ,\eta \right) D_{\mathcal{H}Y}\left(
\rho ,\eta \right) D_{\mathcal{H}X}\mathcal{V}Z \\
=\left( \rho ,\eta ,h\right) \mathbb{R}\left( \mathcal{H}X,\mathcal{H}%
Y\right) \mathcal{V}Z+\left( \rho ,\eta \right) D_{\mathcal{H}\left[
\mathcal{H}X,\mathcal{H}Y\right] _{\left( \rho ,\eta \right) TE}}\mathcal{V}Z
\\
+\left( \rho ,\eta \right) D_{\mathcal{V}\left[ \mathcal{H}X,\mathcal{H}Y%
\right] _{\left( \rho ,\eta \right) TE}}\mathcal{V}Z,%
\end{array}
\\
\begin{array}{l}
\left( \rho ,\eta \right) D_{\mathcal{V}X}\left( \rho ,\eta \right) D_{%
\mathcal{H}Y}\mathcal{V}Z-\left( \rho ,\eta \right) D_{\mathcal{H}Y}\left(
\rho ,\eta \right) D_{\mathcal{V}X}\mathcal{V}Z \\
=\left( \rho ,\eta ,h\right) \mathbb{R}\left( \mathcal{V}X,\mathcal{H}%
Y\right) \mathcal{V}Z+\left( \rho ,\eta \right) D_{h\left[ \mathcal{V}X,%
\mathcal{H}Y\right] _{\left( \rho ,\eta \right) TE}}\mathcal{V}Z \\
+\left( \rho ,\eta \right) D_{\mathcal{V}\left[ \mathcal{V}X,\mathcal{H}Y%
\right] _{\left( \rho ,\eta \right) TE}}\mathcal{V}Z,%
\end{array}
\\
\begin{array}{l}
\left( \rho ,\eta \right) D_{\mathcal{V}X}\left( \rho ,\eta \right) D_{%
\mathcal{V}Y}\mathcal{V}Z-\left( \rho ,\eta \right) D_{\mathcal{V}Y}\left(
\rho ,\eta \right) D_{\mathcal{V}X}\mathcal{V}Z \\
=\left( \rho ,\eta ,h\right) \mathbb{R}\left( \mathcal{V}X,\mathcal{V}%
Y\right) \mathcal{V}Z+\left( \rho ,\eta \right) D_{\mathcal{V}\left[
\mathcal{V}X,\mathcal{V}Y\right] _{\left( \rho ,\eta \right) TE}}\mathcal{V}%
Z.%
\end{array}%
\end{array}%
\right.
\end{equation}
\end{theorem}
\begin{proof}
Using the definition of $%
\left( \rho ,\eta ,h\right) $-curvature associated to the
distinguished linear $\left( \rho ,\eta \right) $-connection $\left(
\left( \rho ,\eta \right) H,\left( \rho ,\eta \right) V\right) $,
results the assertion.
\end{proof}
Using Theorem (\ref{fr9}), the horizontal and vertical sections of adapted
base and an arbitrary section
\begin{equation*}
\begin{array}{c}
Y^{a}\tilde{\delta}_a+\dot{\tilde{Y}}^{a}\dot{\tilde{\partial}}_a\in
\Gamma \left( \left( \rho ,\eta \right) TE,\left( \rho
,\eta \right) \tau _{E},E\right) ,%
\end{array}%
\end{equation*}%
arise to the following Theorem.
\begin{theorem}
The formulas of Ricci type
\begin{equation}  \label{fr1}
\left\{
\begin{array}{l}
\begin{array}{cl}
Y_{~\ |c|b}^{a}-Y_{~\ |b|c}^{a} & =\left( \rho ,\eta ,h\right)
\mathbb{R}_{~e~cb}^{a}Y^{e}+\left( L_{bc}^{e}\circ h\circ
\pi \right) Y_{~\ |e}^{a}\vspace*{2mm} \\
& +\left( \rho ,\eta ,h\right) \mathbb{\tilde{T}}_{~bc}^{e}Y%
^{a}|_{e}+\left( \rho ,\eta ,h\right) \mathbb{T}_{~bc}^{e}Y_{~\
|e}^{\alpha },%
\end{array}%
\vspace*{2mm} \\
\begin{array}{cl}
Y_{~\ |c}^{a}|_{b}-Y^{a}|_{b}{}_{|c} & =\left( \rho ,\eta ,h\right)
\mathbb{P}_{~e~cb}^{a}Y^{e}-\left( \rho ,\eta ,h\right)
\mathbb{\tilde{P}}_{~cb}^{e}Y^{a}\vspace*{2mm}|_{e} \\
& -\left( \rho ,\eta \right) \mathbb{H}_{bc}^{e}Y^{a}|_{e},%
\end{array}%
\vspace*{2mm} \\
\begin{array}{cc}
Y^{a}|_{c}|_{b}-Y^{a}|_{b}|_{c} & =\left( \rho ,\eta ,h\right)
\mathbb{S}_{~e~cb}^{a}Y^{e}+\left( \rho ,\eta ,h\right)
\mathbb{S}_{~bc}^{e}Y^{a}|_{e},%
\end{array}%
\end{array}%
\right.
\end{equation}
and
\begin{equation}  \label{fr2}
\left\{
\begin{array}{l}
\begin{array}{cl}
\dot{\tilde{Y}}_{~\ |c|b}^{a}-\dot{\tilde{Y}}_{~\ |b|c}^{a} & =\left( \rho ,\eta ,h\right) \mathbb{%
\tilde{R}}_{~e~cb}^{a}\dot{\tilde{Y}}^{e}+\left( L_{bc}^{e}\circ
h\circ \pi \right) \dot{\tilde{Y}}_{~\
|e}^{a}\vspace*{2mm} \\
& +\left( \rho ,\eta \right)
\mathbb{\tilde{T}}_{~bc}^{e}\dot{\tilde{Y}}^{a}|_{e}+\left(
\rho ,\eta ,h\right) \mathbb{T}_{~bc}^{e}\dot{\tilde{Y}}_{~\ |e}^{a},%
\end{array}%
\vspace*{2mm} \\
\begin{array}{cl}
\dot{\tilde{Y}}_{~\ |\gamma
}^{a}|_{b}-\dot{\tilde{Y}}^{a}|_{b}{}_{|\gamma } & =\left( \rho
,\eta ,h\right)
\mathbb{\tilde{P}}_{~e~cb}^{a}\dot{\tilde{Y}}^{e}-\left( \rho ,\eta
,h\right)
\mathbb{\tilde{P}}_{~cb}^{e}\dot{\tilde{Y}}^{a}\vspace*{2mm}|_{e} \\
& -\left( \rho ,\eta \right) \mathbb{\tilde{H}}_{bc}^{e}\dot{\tilde{Y}}^{a}|_{e},%
\end{array}%
\vspace*{2mm} \\
\begin{array}{cc}
\dot{\tilde{Y}}^{a}|_{c}|_{b}-\dot{\tilde{Y}}^{a}|_{b}|_{c} & =\left( \rho ,\eta ,h\right) \mathbb{%
\tilde{S}}_{~e~cb}^{a}\dot{\tilde{Y}}^{b}+\left( \rho ,\eta ,h\right) \mathbb{S}%
_{~bc}^{e}\dot{\tilde{Y}}^{a}|_{e},
\end{array}%
\end{array}%
\right.
\end{equation}%
are hold. In particular, if $\left( \rho ,\eta ,h\right) =\left(
Id_{TM},Id_{M},id_{M}\right) $ and the Lie bracket $\left[ ,\right] _{TM}$
is the usual Lie bracket, then the formulas of Ricci type (\ref{fr1}) and (%
\ref{fr2}) reduce to
\begin{equation}
\left\{
\begin{array}{cl}
Y_{~~|k|j}^{i}-Y_{~~|j|k}^{i} & =\mathbb{R}_{~h~kj}^{i}%
Y^{h}+\mathbb{\tilde{T}}_{~jk}^{h}Y^{i}|_{h}+\mathbb{T}%
_{~jk}^{h}Y_{~\ |h}^{i},\vspace*{2mm} \\
Y^{i}~_{|k}|_{j}-Y^{i}|_{j|k} & =\mathbb{P}_{~h~kj}^{i}%
Y^{h}-\mathbb{\tilde{P}}_{~kj}^{h}Y^{i}|_{h}-\mathbb{\tilde{H%
}}_{jk}^{h}Y^{i}|_{h},\vspace*{2mm} \\
Y^{i}|_{c}|_{b}-Y^{i}|_{b}|_{c} & =\mathbb{S}_{~h~kj}^{i}%
Y^{h}+\mathbb{S}_{~jk}^{h}Y^{i}|_{h},%
\end{array}%
\right.
\end{equation}%
and
\begin{equation}
\left\{
\begin{array}{cl}
\dot{\tilde{Y}}_{~~|k|j}^{i}-\dot{\tilde{Y}}_{~~|j|k}^{i} & =\mathbb{R}_{~h~kj}^{i}\dot{\tilde{Y}}^{h}+\mathbb{\tilde{T%
}}_{~jk}^{h}\dot{\tilde{Y}}^{a}|_{h}+\mathbb{T}_{~jk}^{h}\dot{\tilde{Y}}_{~\ |h}^{a},\vspace*{2mm} \\
\dot{\tilde{Y}}_{~~\ |k}^{i}|_{j}-\dot{\tilde{Y}}^{i}|_{j|k} & =\mathbb{\tilde{P}}_{~h~kj}^{i}\dot{\tilde{Y}}^{h}-%
\mathbb{P}_{~kj}^{h}\dot{\tilde{Y}}^{i}|_{h}-\mathbb{H}_{jk}^{h}\dot{\tilde{Y}}^{i}|_{h},\vspace*{2mm}
\\
\dot{\tilde{Y}}^{i}|_{k}|_{j}-\dot{\tilde{Y}}^{i}|_{j}|_{k} & =\mathbb{\tilde{S}}_{~h~kj}^{i}\dot{\tilde{Y}}^{h}+%
\mathbb{S}_{~jk}^{h}\dot{\tilde{Y}}^{i}|_{h}.%
\end{array}%
\right.
\end{equation}
\end{theorem}
Using the $1$-forms associated to distinguished linear $( \rho ,\eta
) $-connection $(( \rho ,\eta ) H,$ $( \rho ,\eta
) V) $%
\begin{equation}
\begin{array}{c}
\left( \rho ,\eta \right) \omega _{b}^{a}=\left( \rho ,\eta \right)
H_{bc}^{a}d\tilde{x}^{c}+\left( \rho ,\eta \right) V_{bc}^{a}\delta \tilde{y}%
^{c},\vspace*{2mm} \\
\left( \rho ,\eta \right) \tilde{\omega}_{b}^{a}=\left( \rho ,\eta \right)
\tilde{H}_{bc}^{a}d\tilde{x}^{c}+\left( \rho ,\eta \right) \tilde{V}%
_{bc}^{a}\delta \tilde{y}^{c},%
\end{array}%
\end{equation}%
the torsion $2$-forms
\begin{equation}
\left\{
\begin{array}{cl}
\left( \rho ,\eta ,h\right) \mathbb{T}^{a} &
\displaystyle=\frac{1}{2}\left( \rho ,\eta ,h\right)
\mathbb{T}_{~bc}^{a}d\tilde{x}^{b}\wedge d\tilde{x}^{c}+\left( \rho
,\eta ,h\right) \mathbb{P}_{~bc}^{a}d\tilde{x}^{b}\wedge
\delta \tilde{y}^{c},\vspace*{2mm} \\
\left( \rho ,\eta ,h\right) \mathbb{\tilde{T}}^{a} & =\displaystyle\frac{1}{2%
}\left( \rho ,\eta ,h\right) \mathbb{\tilde{T}}_{~bc}^{a}d\tilde{x}^{b}\wedge d\tilde{x}^{c}+\left( \rho ,\eta ,h\right) \mathbb{\tilde{P}}%
_{~bc}^{a}d\tilde{x}^{b}\wedge \delta \tilde{y}^{c}\vspace*{2mm} \\
& +\displaystyle\frac{1}{2}\left( \rho ,\eta ,h\right) \mathbb{S}%
_{~bc}^{a}\delta \tilde{y}^{b}\wedge \delta \tilde{y}^{c},
\end{array}%
\right.
\end{equation}%
and the curvature $2$-forms%
\begin{equation}
\left\{
\begin{array}{cl}
\left( \rho ,\eta ,h\right) \mathbb{R}_{~b}^{a} & =\displaystyle\frac{1}{2}%
\left( \rho ,\eta ,h\right) \mathbb{R}_{~b~cd}^{a}d\tilde{x}^{c}\wedge d%
\tilde{x}^{d}+\left( \rho ,\eta ,h\right)
\mathbb{P}_{~b~cd}^{a}d\tilde{x}^{c}\wedge \delta \tilde{y}^{d}\vspace*{2mm} \\
& \displaystyle+\frac{1}{2}\left( \rho ,\eta ,h\right) \mathbb{S}%
_{~b~cd}^{a}\delta \tilde{y}^{c}\wedge \delta \tilde{y}^{d},\vspace*{2mm} \\
\left( \rho ,\eta ,h\right) \mathbb{R}_{~b}^{a} & \displaystyle=\frac{1}{2}%
\left( \rho ,\eta ,h\right) \mathbb{\tilde{R}}_{~b~cd}^{a}d\tilde{x}%
^{c}\wedge d\tilde{x}^{d}+\left( \rho ,\eta ,h\right) \mathbb{\tilde{P}}%
_{~b~cd}^{a}d\tilde{x}^{c}\wedge \delta \tilde{y}^{d}\vspace*{2mm} \\
& \displaystyle+\frac{1}{2}\left( \rho ,\eta ,h\right) \mathbb{\tilde{S}}%
_{~b~cd}^{a}\delta \tilde{y}^{c}\wedge \delta \tilde{y}^{d},%
\end{array}%
\right.
\end{equation}%
The following will be obtained.
\begin{theorem}
\label{fr5} The identities of Cartan type
\begin{equation}
\begin{array}{c}
\left( \rho ,\eta ,h\right) \mathbb{T}^{a}=d^{\left( \rho ,\eta
\right) TE}\left( d\tilde{x}^{a}\right) +\left( \rho ,\eta \right)
\omega
_{b}^{a}\wedge d\tilde{x}^{b},\vspace*{2mm} \\
\left( \rho,\eta ,h\right) \mathbb{\tilde{T}}^{a}=d^{\left( \rho
,\eta \right) TE}\left( \delta \tilde{y}^{a}\right) +\left( \rho
,\eta \right)
\tilde{\omega}_{b}^{a}\wedge \delta \tilde{y}^{b},%
\end{array}
\label{fr3}
\end{equation}%
and
\begin{equation}
\begin{array}{c}
\left( \rho ,\eta ,h\right) \mathbb{R}_{~b}^{a}=d^{\left( \rho ,\eta \right)
TE}\left( \left( \rho ,\eta \right) \omega _{b}^{a}\right) +\left( \rho
,\eta \right) \omega _{c}^{a}\wedge \left( \rho ,\eta \right) \omega
_{b}^{c},\vspace*{2mm} \\
\left( \rho ,\eta ,h\right) \mathbb{\tilde{R}}_{~b}^{a}=d^{\left( \rho ,\eta
\right) TE}\left( \left( \rho ,\eta \right) \tilde{\omega}_{b}^{a}\right)
+\left( \rho ,\eta \right) \tilde{\omega}_{c}^{a}\wedge \left( \rho ,\eta
\right) \tilde{\omega}_{b}^{c},
\end{array}
\label{fr4}
\end{equation}
are true.
\end{theorem}
\begin{rem}
\label{fr6} For any $X,Y,Z\in \Gamma \left( \left( \rho ,\eta \right)
TE,\left( \rho ,\eta \right) \tau _{E},E\right) $, the identities
\begin{equation}
\begin{array}{rc}
\mathcal{V}\left( \rho ,\eta ,h\right) \mathbb{R}\left( X,Y\right) \mathcal{H%
}Z & =0,\vspace*{2mm} \\
\mathcal{H}\left( \rho ,\eta ,h\right) \mathbb{R}\left( X,Y\right) \mathcal{V%
}Z & =0,%
\end{array}%
\end{equation}%
\begin{equation}
\begin{array}{cc}
\mathcal{V}D_{X}\left( \left( \rho ,\eta ,h\right) \mathbb{R}\left(
Y,Z\right) \mathcal{H}U\right) & =0,\vspace*{2mm} \\
\mathcal{H}D_{X}\left( \left( \rho ,\eta ,h\right) \mathbb{R}\left(
Y,Z\right) \mathcal{V}U\right) & =0,%
\end{array}%
\end{equation}%
and
\begin{align}
( \rho ,\eta ,h) \mathbb{R}( X,Y) Z&=\mathcal{H}( \rho ,\eta ,h) \mathbb{R}(
X,Y) \mathcal{H}Z+\mathcal{V}( \rho ,\eta ,h) \mathbb{R}( X,Y) \mathcal{V}Z,
&
\end{align}
are hold.
\end{rem}
Using the formulas of Bianchi type from Theorem (\ref{fr5}) and Remark (\ref%
{fr6}), the following will be resulted.
\begin{theorem}
The identities of Bianchi type
\begin{equation}  \label{fr7}
\left\{
\begin{array}{l}
\underset{cyclic\left( X,Y,Z\right) }{\sum }\left\{ \mathcal{H}\left( \rho
,\eta \right) D_{X}\left( \left( \rho ,\eta ,h\right) \mathbb{T}\left(
Y,Z\right) \right) -\mathcal{H}\left( \rho ,\eta ,h\right) \mathbb{R}\left(
X,Y\right) Z\right. \vspace*{2mm} \\
\qquad \qquad +\mathcal{H}\left( \rho ,\eta ,h\right) \mathbb{T}\left(
\mathcal{H}\left( \rho ,\eta ,h\right) \mathbb{T}\left( X,Y\right) ,Z\right)
\vspace*{2mm} \\
\qquad \qquad \left. +\mathcal{H}\left( \rho ,\eta ,h\right) \mathbb{T}%
\left( \mathcal{V}\left( \rho ,\eta ,h\right) \mathbb{T}\left( X,Y\right)
,Z\right) \right\} =0,\vspace*{4mm} \\
\underset{cyclic\left( X,Y,Z\right) }{\sum }\left\{ \mathcal{V}\left( \rho
,\eta \right) D_{X}\left( \left( \rho ,\eta ,h\right) \mathbb{T}\left(
Y,Z\right) \right) -\mathcal{V}\left( \rho ,\eta ,h\right) \mathbb{R}\left(
X,Y\right) Z\right. \vspace*{2mm} \\
\qquad \qquad +\mathcal{V}\left( \rho ,\eta ,h\right) \mathbb{T}\left(
\mathcal{H}\left( \rho ,\eta ,h\right) \mathbb{T}\left( X,Y\right) ,Z\right)
\vspace*{2mm} \\
\qquad \qquad \left. +\mathcal{V}\left( \rho ,\eta ,h\right) \mathbb{T}%
\left( \mathcal{V}\left( \rho ,\eta ,h\right) \mathbb{T}\left( X,Y\right)
,Z\right) \right\} =0,%
\end{array}%
\right.
\end{equation}
and
\begin{equation}  \label{fr8}
\left\{
\begin{array}{l}
\underset{cyclic\left( X,Y,Z,U\right) }{\sum }\left\{ \mathcal{H}\left( \rho
,\eta \right) D_{X}\left( \left( \rho ,\eta ,h\right) \mathbb{R}\left(
Y,Z\right) U\right) \right. \vspace*{2mm} \\
\qquad \qquad -\mathcal{H}\left( \rho ,\eta ,h\right) \mathbb{R}\left(
\mathcal{H}\left( \rho ,\eta ,h\right) \mathbb{T}\left( X,Y\right) ,Z\right)
U\vspace*{2mm} \\
\qquad \qquad \left. -\mathcal{H}\left( \rho ,\eta ,h\right) \mathbb{R}%
\left( \mathcal{V}\left( \rho ,\eta ,h\right) \mathbb{T}\left( X,Y\right)
,Z\right) U\right\} =0,\vspace*{4mm} \\
\underset{cyclic\left( X,Y,Z,U\right) }{\sum }\left\{ \mathcal{V}\left( \rho
,\eta \right) D_{X}\left( \left( \rho ,\eta ,h\right) \mathbb{R}\left(
Y,Z\right) U\right) \right. \vspace*{2mm} \\
\qquad \qquad -\mathcal{V}\left( \rho ,\eta ,h\right) \mathbb{R}\left(
\mathcal{H}\left( \rho ,\eta ,h\right) \mathbb{T}\left( X,Y\right) ,Z\right)
U\vspace*{2mm} \\
\qquad \qquad \left. -\mathcal{V}\left( \rho ,\eta ,h\right) \mathbb{R}%
\left( \mathcal{V}\left( \rho ,\eta ,h\right) \mathbb{T}\left( X,Y\right)
,Z\right) U\right\} =0,%
\end{array}%
\right.
\end{equation}%
for any $X,Y,Z\in \Gamma \left( \left( \rho ,\eta \right) TE,\left( \rho
,\eta \right) \tau _{E},E\right)$ are hold.
\end{theorem}
\begin{cor}
Using the sections $\left( \delta _{\theta },\delta _{\gamma },\delta
_{\beta }\right) $, the identities (\ref{fr7}) become
\begin{equation}
\left\{
\begin{array}{l}
\underset{cyclic\left( \beta ,\gamma ,\theta \right) }{\sum }\left\{ \left(
\rho ,\eta ,h\right) \mathbb{T}_{~~bc_{|d}}^{a}-\left( \rho ,\eta ,h\right)
\mathbb{R}_{~b~cd}^{a}\right. \vspace*{2mm} \\
\qquad \qquad \left. +\left( \rho ,\eta ,h\right) \mathbb{T}_{~cd}^{e}\left(
\rho ,\eta ,h\right) \mathbb{T}_{~be}^{a}+\left( \rho ,\eta ,h\right)
\mathbb{\tilde{T}}_{~cd}^{e}\left( \rho ,\eta ,h\right) \mathbb{\tilde{T}}%
_{~be}^{a}\right\} =0,\vspace*{4mm} \\
\underset{cyclic\left( \beta ,\gamma ,\theta \right) }{\sum }\left\{ \left(
\rho ,\eta ,h\right) \mathbb{T}_{~\ bc_{|d}}^{a}+\left( \rho ,\eta ,h\right)
\mathbb{T}_{~cd}^{e}\left( \rho ,\eta ,h\right) \mathbb{\tilde{P}}%
_{~be}^{a}\right. \vspace*{2mm} \\
\qquad \qquad \left. +\left( \rho ,\eta ,h\right) \mathbb{\tilde{T}}%
_{~cd}^{e}\left( \rho ,\eta ,h\right) \mathbb{\tilde{P}}_{~be}^{a}\right\}
=0,%
\end{array}%
\right.
\end{equation}
and using the sections $\left( \delta _{\lambda },\delta _{\theta },\delta
_{\gamma },\delta _{\beta }\right) $, the identities (\ref{fr8}) become
\begin{equation}
\left\{
\begin{array}{l}
\underset{cyclic\left( \beta ,\gamma ,\theta ,\lambda \right) }{\sum }%
\left\{ \left( \rho ,\eta ,h\right) \mathbb{R}_{~\ bcd_{|e}}^{a}-\left( \rho
,\eta ,h\right) \mathbb{T}_{~de}^{l}\left( \rho ,\eta ,h\right) \mathbb{R}%
_{~b~cl}^{a}\right. \vspace*{2mm} \\
\qquad \qquad \left. -\left( \rho ,\eta ,h\right) \mathbb{\tilde{T}}%
_{~de}^{l}\left( \rho ,\eta ,h\right) \mathbb{P}_{~b~cl}^{a}\right\} =0,%
\vspace*{4mm} \\
\underset{cyclic\left( \beta ,\gamma ,\theta ,\lambda \right) }{\sum }%
\left\{ \left( \rho ,\eta ,h\right) \mathbb{\tilde{R}}_{~\
b~cd_{|e}}^{a}-\left( \rho ,\eta ,h\right) \mathbb{T}_{~de}^{l}\left( \rho
,\eta ,h\right) \mathbb{\tilde{R}}_{~b~cl}^{a}\right. \vspace*{2mm} \\
\qquad \qquad \left. -\left( \rho ,\eta ,h\right) \mathbb{T}_{~de}^{l}\left(
\rho ,\eta ,h\right) \mathbb{P}_{~b~cl}^{a}\right\} =0.%
\end{array}%
\right.
\end{equation}
\end{cor}

Using another base of sections, we shall obtain new identities of Bianchi
type necessary in the applications.
\section{Geodesics for mechanical $(\rho, \eta)$-systems}
In this section we present some aspects about the geometry of a
mechanical $\left( \rho ,\eta \right) $-system $\left( \left( E,\pi
,M\right) ,F_{e},\left( \rho
,\eta \right) \Gamma \right) ,$ where%
\begin{equation}
\begin{array}[t]{l}
F_{e}=F^{a}\dot{\tilde{\partial}}_{a}\in \Gamma \left( V\left( \left( \rho
,\eta \right) TE,\left( \rho ,\eta \right) \tau _{E},E\right) \right),
\end{array}%
\end{equation}%
is an external force and $\left( \rho ,\eta \right) \Gamma $ is a $\left(
\rho ,\eta \right) $-connection for the vector bundle $\left( E,\pi
,M\right)$.
\begin{definition}
For any section $u=u^{a}s_{a},$ we build the canonical section $U=U^{a}%
\tilde{\delta}_{a}+U^{a}\dot{\tilde{\partial}}_{a}\in \Gamma \left(
\left( \rho ,\eta \right) TE,\left( \rho ,\eta \right) \tau
_{E},E\right) $ such that $U^{a}\left( u_{x}\right) =y^a$, for any
$a= 1,\cdots,r$, where $(x^1, \cdots, x^m, y^1, \cdots, y^r)$ are
the components of the point $u_x$ in a vector local $m+r$-chart
$({V}, s_{{V}})$. The vertical section
$\mathbb{C}=U^{a}\dot{\tilde{\partial}}_{a}$
 will be called the Liouville section. Moreover, a section $S\in \Gamma \left( \left( \rho ,\eta \right) TE,\left( \rho
,\eta \right) \tau _{E},E\right) $ will be called $\left( \rho ,\eta \right)
$-semispray if there exists an almost tangent structure $e$ such that
$e\left( S\right) =\mathbb{C}.$
\end{definition}
\begin{remark}\label{GodGod}
Using the above definition it is easy to see that $\partial_iU^b=0$ and $\dot{\partial}_aU^b=\delta_a^b$.
\end{remark}
\begin{example}
Let $g$ be a manifolds morphism on $E$ such that $\left(
g,h\right) $ is a locally invertible vector bundles morphism on $\left(
E,\pi ,M\right) $. Using the almost tangent structure $\mathcal{J}_{\left( g,h\right) },$
the section
\begin{equation*}
\begin{array}[t]{l}
S=\left( g_{b}^{a}\circ h\circ \pi \right) U^{b}\tilde{\partial}_{a}-2\left(
G^{a}-\frac{1}{4}F^{a}\right) \dot{\tilde{\partial}}_{a},%
\end{array}%
\end{equation*}%
is a $\left( \rho ,\eta \right) $-semispray such that the real local
functions $G^{a},\ a\in 1,\cdots ,n,$ satisfy the condition%
\begin{equation*}
\begin{array}{l}
( \rho ,\eta ) \Gamma _{c}^{a}=(\tilde{g}_{c}^{b}\circ
h\circ \pi ) \dot{\partial}_b( G^{a}-\frac{1}{4}F^{a})\\
-\frac{1}{2}( g_{e}^{d}\circ h\circ \pi ) U^{e}(
L_{dc}^{f}\circ h\circ \pi ) ( \tilde{g}_{f}^{a}\circ h\circ \pi
) \\
+\frac{1}{2}( \rho _{c}^{j}\circ h\circ \pi)\partial_j( g_{e}^{b}\circ h\circ \pi )U^{e}( \tilde{%
g}_{b}^{a}\circ h\circ \pi ) \\
-\frac{1}{2}( g_{e}^{b}\circ h\circ \pi ) U^{e}( \rho
_{b}^{i}\circ h\circ \pi )\partial_i( \tilde{g}%
_{c}^{a}\circ h\circ \pi ).
\end{array}%
\end{equation*}
The $\left( \rho ,\eta \right) $-semispray $S$ will be called the canonical $%
\left( \rho ,\eta \right) $-semispray associated to the mechanical $\left( \rho
,\eta \right) $-system $\left( \left( E,\pi ,M\right) ,F_{e},\left( \rho
,\eta \right) \Gamma \right) $ and locally invertible vector bundles
morphism $\left( g,h\right)$ (see \cite{4}).
\end{example}
\begin{defn}
Let $c:I\rightarrow M$ be a differentiable curve and $\dot{c}$ be its $\left( g,h\right)$-lift. If $\frac{d\dot{c}\left( t\right) }{dt}=\Gamma (\tilde{\rho},Id_{E})S(\dot{c}%
(t))$, then the curve $\dot{c}$ is an \emph{integral curve} of the $\left(
\rho ,\eta \right) $-semispray $S$ of the mechanical $\left( \rho ,\eta
\right) $-system $\left( \left( E,\pi ,M\right) ,F_{e},\left( \rho ,\eta
\right) \Gamma \right) $.
\end{defn}
\begin{theorem}\label{God1}
All of the $\left( g,h\right) $-lifts solutions of the equations
\begin{equation*}
\begin{array}[t]{l}
\frac{dy^{a}\left( t\right) }{dt}+2G^{a}\!\circ u\left( c,\dot{c}\right)
\left( \eta \circ h\circ c\left( t\right) \right) {=}\frac{1}{2}F^{a}\!\circ
u\left( c,\dot{c}\right) \left( \eta \circ h\circ c\left( t\right) \right)
\!,\,a{\in }1,\cdots ,r,%
\end{array}%
\end{equation*}%
are integral curves of the canonical $\left( \rho ,\eta \right) $-semispray
associated to mechanical $\left( \rho ,\eta \right) $-system $\left( \left(
E,\pi ,M\right) ,F_{e},\left( \rho ,\eta \right) \Gamma \right) $ and
locally invertible vector bundles morphism $\left( g,h\right) .$
\end{theorem}
\begin{defn}
If $S$\ is a $\left( \rho ,\eta \right) $-semispray, then the section $%
\left[ \mathbb{C},S\right] _{(\rho, \eta)TE}-S$ will be called the derivation of $%
\left( \rho ,\eta \right) $-semispray $S$. The $\left( \rho ,\eta \right) $-semispray $S$ will be called $\left( \rho
,\eta \right) $-spray if $S\circ 0$ is differentiable of class $C^{1}$\ where $0$ is the null
section, and its derivation is the null section.
\end{defn}
\begin{lemma}
If $S$ is the canonical $\left( \rho ,\eta \right) $-semispray associated to
the mechanical $\left( \rho ,\eta \right) $-system $\left( \left( E,\pi
,M\right) ,F_{e},\left( \rho ,\eta \right) \Gamma \right) $ and locally invertible vector
bundles morphism $\left( g,h\right) $, then $S$ is $\left( \rho ,\eta \right) $-spray if and only if
\begin{align}\label{IM}
U^b\dot{\partial}_b(G^a-\frac{1}{4}F^a)=2(G^a-\frac{1}{4}F^a).
\end{align}
\end{lemma}
\begin{proof}
Using the locally expressions of $\mathbb{C}$ and $S$ and considering $\tilde{\partial}_aU^b=0$ and $\dot{\tilde{\partial}}_aU^b=\delta_a^b$, yield
\begin{align*}
\left[ \mathbb{C}, S\right] _{(\rho, \eta)TE}-S&=[U^b\dot{\title{\partial}}_b, \left( g_{c}^{a}\circ h\circ \pi \right) U^{c}\tilde{\partial}_{a}-2(
G^{a}-\frac{1}{4}F^{a}) \dot{\tilde{\partial}}_{a}] _{(\rho, \eta)TE}\\
&\ \ \ -\left( g_{c}^{a}\circ h\circ \pi \right) U^{c}\tilde{\partial}_{a}+2(
G^{a}-\frac{1}{4}F^{a}) \dot{\tilde{\partial}}_{a}\\
&=(g_{b}^{a}\circ h\circ \pi)U^{b}\tilde{\partial}_{a}-2U^b\dot{\tilde{\partial}}_b(
G^{a}-\frac{1}{4}F^{a})\dot{\tilde{\partial}}_{a}+2(
G^{a}-\frac{1}{4}F^{a}) \dot{\tilde{\partial}}_{a}\\
&\ \ \ -\left( g_{c}^{a}\circ h\circ \pi \right) U^{c}\tilde{\partial}_{a}+2(
G^{a}-\frac{1}{4}F^{a}) \dot{\tilde{\partial}}_{a}\\
&=2\Big(2(G^{a}-\frac{1}{4}F^{a})-U^b\dot{\tilde{\partial}}_b(
G^{a}-\frac{1}{4}F^{a})\Big)\dot{\tilde{\partial}}_{a},
\end{align*}
completing the proof.
\end{proof}
\begin{theorem}\label{7}
 If $S$ is the canonical $\left( \rho ,\eta \right) $-spray
associated to the mechanical $\left( \rho ,\eta \right) $-system $\left( \left(
E,\pi ,M\right) ,F_{e},\left( \rho ,\eta \right) \Gamma \right) $ and
locally invertible vector bundles morphism $\left( g,h\right) $, then
\begin{equation*}
\begin{array}{cl}
2\left( G^{a}-\frac{1}{4}F^{a}\right) & \!\!\!\!=\left( \rho ,\eta \right)
\Gamma _{c}^{a}(g_{f}^{c}\circ h\circ \pi )U^{f} \\
& \ +\frac{1}{2}(g_{e}^{d}\circ h\circ \pi )U^{e}(L_{dc}^{b}\circ h\circ \pi
)(\tilde{g}_{b}^{a}\circ h\circ \pi )(g_{f}^{c}\circ h\circ \pi )U^{f} \\
& \ -\frac{1}{2}(\rho _{c}^{j}\circ h\circ \pi )\partial_j
(g_{e}^{b}\circ h\circ \pi )U^{e}(\tilde{g}_{b}^{a}\circ
h\circ \pi )(g_{f}^{c}\circ h\circ \pi )U^{f} \\
& \ +\frac{1}{2}(g_{e}^{b}\circ h\circ \pi )U^{e}(\rho _{b}^{i}\circ h\circ
\pi )\partial_i(\tilde{g}_{c}^{a}\circ h\circ \pi )
(g_{f}^{c}\circ h\circ \pi )U^{f}.%
\end{array}%
\end{equation*}
\end{theorem}
The spray in Theorem (\ref{7}) will be called the canonical $\left( \rho
,\eta \right) $-spray associated to the mechanical $\left( \rho ,\eta \right) $%
-system $\left( \left( E,\pi ,M\right) ,F_{e},v\Gamma \right) $ and from
locally invertible vector bundles morphism $(g,h)$.
\begin{remark}
If $\left( \rho ,\eta \right) =\left( Id_{TM},Id_{M}\right) $
and $\left( g,h\right) =\left( Id_{TM},Id_{M}\right) ,$ then we get the
canonical spray associated to the connection $\Gamma $ which is similar with the
classical canonical spray associated to the connection $\Gamma$.
\end{remark}
\begin{theorem}
All of the $\left( g,h\right) $-lifts solutions of the system of equations
\begin{align*}
& \frac{dy^{a}}{dt}+\left( \rho ,\eta \right) \Gamma _{c}^{a}(g_{f}^{c}\circ
h\circ \pi )U^{f} \\
& \displaystyle+\frac{1}{2}(g_{e}^{d}\circ h\circ \pi )U^{e}(L_{dc}^{b}\circ
h\circ \pi )(\tilde{g}_{b}^{a}\circ h\circ \pi )(g_{f}^{c}\circ h\circ \pi
)U^{f} \\
& -\frac{1}{2}(\rho _{c}^{j}\circ h\circ \pi )\partial_j(g_{e}^{b}\circ
h\circ \pi )U^{e}(\tilde{g}_{b}^{a}\circ h\circ \pi
)(g_{f}^{c}\circ h\circ \pi )U^{f} \\
& +\frac{1}{2}(g_{e}^{b}\circ h\circ \pi )U^{e}(\rho _{b}^{i}\circ h\circ
\pi )\partial_i(\tilde{g}_{c}^{a}\circ h\circ \pi )
(g_{f}^{c}\circ h\circ \pi )U^{f}=0,
\end{align*}%
are the integral curves of the canonical $\left( \rho ,\eta \right) $-spray
associated to the mechanical $\left( \rho ,\eta \right) $-system $\left( \left(
E,\pi ,M\right) ,F_{e},\left( \rho ,\eta \right) \Gamma \right) $ and
locally invertible vector bundles morphism $\left( g,h\right) .$
\end{theorem}
\begin{definition}
If $c:I\rightarrow M$ is a differentiable curve such that its $\left( g,h\right) $-lift
$\dot{c}$
is integral curve for the $\left( \rho ,\eta \right) $-spray $S$ of the mechanical $\left( \rho ,\eta \right) $-system $%
\left( \left( E,\pi ,M\right) ,F_{e},\left( \rho ,\eta \right)
\Gamma \right) $, then the curve $c$ will be called the geodesic
for the mechanical $\left( \rho ,\eta \right) $-system
$\left( \left( E,\pi ,M\right) ,F_{e},\left( \rho ,\eta \right)
\Gamma \right) $ with respect to $S$.
\end{definition}
Using (\ref{God}), Theorem \ref{God1} and the above definition we deduce that the curve $c:I\rightarrow M$ is geodesic, if for $i\in\{1, \cdots, m\}$, $a, b\in\{1, \cdots, r\}$,
\begin{equation}
\left\{
\begin{array}{cc}
&\hspace{-5mm}\rho^i_b(\eta\circ h\circ c(t))g^b_a(h\circ c(t))y^a(t)=\frac{d(\eta\circ h\circ c)^i(t)}{dt},\\\\
&\hspace{-2.5cm}\frac{dy^a}{dt}+2(G^a-\frac{1}{4}F^a)(\dot{c}(t))=0.
\end{array}
\right.
\end{equation}
\section{Projectively related and Weyl type theorems}
Let $\left( \rho ,\eta \right) \nabla $ be the Berwald covariant
$\left( \rho ,\eta \right) $-derivative. Then using (\ref{good}) for
any
$X=X^{a}\tilde{\delta}_{a}+\dot{\tilde{X}}^{a}\dot{\tilde{\partial}}_{a}\in
\Gamma (\!\left( \rho ,\eta \right) TE,\!\left( \rho ,\eta \right)
\tau _{E},\!E)$ and $T\in \mathcal{T}_{qs}^{pr}\!(\!\left( \rho
,\eta \right) TE,\!\left( \rho ,\eta \right) \tau _{E},\!E)$ we have
\begin{equation*}
\begin{array}{l}
\left( \rho ,\eta \right) \nabla _{X}\left(
T_{b_{1}\cdots b_{q}e_{1}\cdots e_{s}}^{a_{1}\cdots a_{p}d_{1}\cdots d_{r}}\tilde{\delta}%
_{a_{1}}\otimes \cdots \otimes \tilde{\delta}_{a_{p}}\otimes d\tilde{x}%
^{b_{1}}\otimes \cdots \otimes \right. \vspace*{1mm} \\
\hspace*{9mm}\left. \otimes d\tilde{x}^{b_{q}}\otimes \dot{\tilde{\partial}}%
_{d_{1}}\otimes\cdots \otimes \dot{\tilde{\partial}}_{d_{r}}\otimes
\delta
\tilde{y}^{e_{1}}\otimes\cdots \otimes \delta \tilde{y}^{e_{s}}\right) =%
\vspace*{1mm} \\
\hspace*{9mm}=X^{c}T_{b_{1}\cdots b_{q}e_{1}\cdots e_{s}\mid
c}^{a_{1}\cdots a_{p}d_{1}\cdots d_{r}}\tilde{\delta}_{a_{1}}\otimes
\cdots \otimes
\tilde{\delta}_{a_{p}}\otimes d\tilde{x}^{b_{1}}\otimes\cdots \otimes d\tilde{x}%
^{b_{q}}\otimes \dot{\tilde{\partial}}_{d_{1}}\otimes \cdots \otimes \vspace*{1mm%
} \\
\hspace*{9mm}\otimes \dot{\tilde{\partial}}_{d_{r}}\otimes \delta \tilde{y}%
^{e_{1}}\otimes\cdots \otimes \delta \tilde{y}^{e_{s}}+\dot{\tilde{X}}%
^{c}T_{b_{1}\cdots b_{q}e_{1}\cdots e_{s}}^{a_{1}\cdots a_{p}d_{1}\cdots d_{r}}\mid _{c}%
\tilde{\delta}_{a_{1}}\otimes \cdots \otimes \vspace*{1mm} \\
\hspace*{9mm}\otimes \tilde{\delta}_{a_{p}}\otimes
d\tilde{x}^{b_{1}}\otimes\cdots \otimes d\tilde{x}^{b_{q}}\otimes
\dot{\tilde{\partial}}_{d_{1}}\otimes
\cdots \otimes \dot{\tilde{\partial}}_{d_{r}}\otimes \delta \tilde{y}%
^{e_{1}}\otimes\cdots \otimes \delta \tilde{y}^{e_{s}},%
\end{array}%
\end{equation*}%
where
\begin{equation*}
\begin{array}{l}
T_{b_{1}\cdots b_{q}e_{1}\cdots e_{s}\mid c}^{a_{1}\cdots a_{p}d_{1}\cdots d_{r}}=\vspace*{%
2mm}\Gamma \left( \tilde{\rho},Id_{E}\right) \left( \tilde{\delta}%
_{c}\right) T_{b_{1}\cdots b_{q}e_{1}\cdots e_{s}}^{a_{1}\cdots a_{p}d_{1}\cdots d_{r}} \\
\hspace*{8mm}+\dot{\partial}_{a}\left( \rho ,\eta \right) \Gamma
_{c}^{a_{1}}T_{b_{1}\cdots b_{q}e_{1}\cdots e_{s}}^{aa_{2}\cdots a_{p}d_{1}\cdots d_{r}}+\cdots +%
\vspace*{2mm}\dot{\partial}_{a}\left( \rho ,\eta \right) \Gamma
_{c}^{a_{p}}T_{b_{1}\cdots b_{q}e_{1}\cdots e_{s}}^{a_{1}\cdots
a_{p-1}ad_{1}\cdots d_{r}}
\\
\hspace*{8mm}-\dot{\partial}_{b_{1}}\left( \rho ,\eta \right) \Gamma
_{c}^{b}T_{bb_{2}\cdots b_{q}e_{1}\cdots e_{s}}^{a_{1}\cdots a_{p}d_{1}\cdots d_{r}}-\cdots -%
\vspace*{2mm}\dot{\partial}_{b_{q}}\left( \rho ,\eta \right) \Gamma
_{c}^{b}T_{b_{1}\cdots b_{q-1}be_{1}\cdots e_{s}}^{a_{1}\cdots a_{p}d_{1}\cdots d_{r}} \\
\hspace*{8mm}+\dot{\partial}_{d}\left( \rho ,\eta \right) \Gamma
_{c}^{d_{1}}T_{b_{1}\cdots b_{q}e_{1}\cdots e_{s}}^{a_{1}\cdots a_{p}dd_{2}\cdots d_{r}}+\cdots +%
\vspace*{2mm}\dot{\partial}_{d}\left( \rho ,\eta \right) \Gamma
_{c}^{d_{r}}\left( \rho ,\eta \right) T_{b_{1}\cdots
b_{q}e_{1}\cdots e_{s}}^{a_{1}\cdots a_{p}d_{1}\cdots d_{r-1}d}
\\
\hspace*{8mm}-\dot{\partial}_{e_{1}}\left( \rho ,\eta \right) \Gamma
_{c}^{e}T_{b_{1}\cdots b_{q}ee_{2}\cdots e_{s}}^{a_{1}\cdots a_{p}d_{1}\cdots d_{r}}-%
\vspace*{2mm}\cdots -\dot{\partial}_{e_{s}}\left( \rho ,\eta \right)
\Gamma _{c}^{e}T_{b_{1}\cdots b_{q}e_{1}\cdots
e_{s-1}e}^{a_{1}\cdots a_{p}d_{1}\cdots d_{r}},
\end{array}
\end{equation*}%
and
\begin{equation*}
\begin{array}{l}
T_{b_{1}\cdots b_{q}e_{1}\cdots e_{s}}^{a_{1}\cdots a_{p}d_{1}\cdots d_{r}}\mid _{c}=%
\vspace*{2mm}\Gamma \left( \tilde{\rho},Id_{E}\right) \left( \dot{\tilde{%
\partial}}_{c}\right)
T_{b_{1}\cdots b_{q}e_{1}\cdots e_{s}}^{a_{1}\cdots a_{p}d_{1}\cdots
d_{r}}.
\end{array}%
\end{equation*}
Now we define the $v$-covariant $\left( \rho ,\eta
\right)$-derivative $\left( \rho ,\eta \right) \nabla _{X}^{v}$ as
\begin{equation*}
\left( \rho ,\eta \right) \nabla _{X}^{v}f:=\Gamma \left( \tilde{\rho}%
,Id_{E}\right)
(\mathcal{V}X)f=\tilde{X}^{a}\dot{{\partial}}_{a}\left( f\right),
\end{equation*}
and
\begin{equation*}
\begin{array}{l}
\left( \rho ,\eta \right) \nabla _{X}^{v}\left(
T_{b_{1}\cdots b_{q}e_{1}\cdots e_{s}}^{a_{1}\cdots a_{p}d_{1}\cdots d_{r}}\tilde{\delta}%
_{a_{1}}\otimes \cdots \otimes \tilde{\delta}_{a_{p}}\otimes d\tilde{x}%
^{b_{1}}\otimes \cdots \otimes \right. \vspace*{1mm} \\
\hspace*{9mm}\left. \otimes d\tilde{x}^{b_{q}}\otimes \dot{\tilde{\partial}}%
_{d_{1}}\otimes \cdots \otimes \dot{\tilde{\partial}}_{d_{r}}\otimes
\delta
\tilde{y}^{e_{1}}\otimes \cdots \otimes \delta \tilde{y}^{e_{s}}\right) = \\
\hspace*{9mm}\dot{\tilde{X}}%
^{c}T_{b_{1}\cdots b_{q}e_{1}\cdots e_{s}}^{a_{1}\cdots a_{p}d_{1}\cdots d_{r}}\mid _{c}%
\tilde{\delta}_{a_{1}}\otimes \cdots \otimes \vspace*{1mm}\tilde{\delta}%
_{a_{p}}\otimes d\tilde{x}^{b_{1}}\otimes \cdots \otimes d\tilde{x}^{b_{q}} \\
\hspace*{9mm}\otimes \dot{\tilde{\partial}}_{d_{1}}\otimes \cdots \otimes \dot{%
\tilde{\partial}}_{d_{r}}\otimes \delta \tilde{y}^{e_{1}}\otimes
\cdots \otimes
\delta \tilde{y}^{e_{s}},%
\end{array}%
\end{equation*}
where $f\in \mathcal{F}(E)$.
\begin{definition}
Hessian of a smooth function $f$ on $E$ is a $\left( _{0,2}^{0,
0}\right) $-tensor on $((\rho ,\eta )TE,(\rho ,\eta )\tau _{E},E)$
defined by
\begin{align*}
\text{Hess}(f)&=( \dot{\partial}_{a}\dot{\partial}%
_{b}f) \delta \tilde{y}^{a}\otimes \delta \tilde{y}^{b}.
\end{align*}
\end{definition}

\begin{definition}
Function $f\in \mathcal{F}(E)$ is called homogeneous of degree
one or 1-homogeneous, if $(\rho ,\eta )\nabla _{U}^{v}f=f$.
\end{definition}
Since $\mathcal{V}U=\mathbb{C}$, then $(\rho ,\eta )\nabla _{U}^{v}f=%
\vspace*{2mm}\Gamma \left( \tilde{\rho},Id_{E}\right) (\mathbb{C})f=U^{a}%
\dot{\partial}_{a}f $. Thus the 1-homogeneity of $F$ is equivalent
to the condition
\begin{equation}\label{esi}
U^{a}\dot{\partial}_{a}f=f.
\end{equation}
\begin{lemma}
Let $f\in \mathcal{F}(E)$ be 1-homogeneous. Then the following is
hold.
\begin{equation*}
(\rho, \eta)\nabla _{U}^{v}(\text{Hess}f)=-\text{Hess}f.
\end{equation*}
\end{lemma}
\begin{proof}
Using the definition of the Hessian, leads to
\[
(\rho, \eta)\nabla
_{U}^{v}(\text{Hess}f)=U^c(\dot{\partial}_a\dot{\partial}_bf)\mid_c\delta\tilde{y}^{a}\otimes
\delta
\tilde{y}^{b}=U^c(\dot{\partial}_c\dot{\partial}_a\dot{\partial}_bf)\delta\tilde{y}^{a}\otimes
\delta \tilde{y}^{b}.
\]
Since $f$ is 1-homogeneous and $\dot{\partial}_b(U^a)=\delta_b^a$,
then using (\ref{esi}) get
$U^c\dot{\partial}_c\dot{\partial}_a\dot{\partial}_bf=-\dot{\partial}_a\dot{\partial}_bf$.
Using preceding result, the proof is concluded.
\end{proof}
\begin{definition}
A horizontal projector $\mathcal{H}$ is said homogeneous, if $(\rho, \eta)\nabla _{{%
\mathcal{H}}X}U=0$, where $(\rho, \eta)\nabla $ denotes the Berwald derivative and $X$
is arbitrary in\\ $\Gamma ((\rho ,\eta )TE,(\rho ,\eta )\tau _{E},E)$.
\end{definition}

\begin{lemma}\label{Lemma7.5}
A horizontal projector $\mathcal{H}$ is homogeneous if and only if
\begin{equation}
U^c{{\dot{{\partial}}}_c}((\rho,\eta)\Gamma^b_a)=(\rho,\eta)\Gamma^b_a.
\end{equation}
\end{lemma}

\begin{proof}
Since  $\mathcal{H}$ is homogeneous, then  $\nabla_{{\mathcal{H}}X}U=0$. Setting $X=\dot{\tilde{\delta}}_a$ and the locally expression of $U$ in this equation yields
\[
(\rho^i_a\circ h\circ\pi)\partial_iU^b-(\rho, \eta)\Gamma^c_a\dot{\partial}_cU^b+U^c\dot{\partial}_c((\rho, \eta)\Gamma^b_a)=0.
\]
Since $\partial_iU^b=0$ and $\dot{\partial}_cU^b=\delta^b_c$, then the above equation reduces to
\begin{equation}\label{esi1}
-(\rho, \eta)\Gamma^b_a+U^c\dot{\partial}_c((\rho, \eta)\Gamma^b_a)=0,
\end{equation}
which completes the proof.
\end{proof}
\begin{proposition}\label{Prop7.6}
Let $\overset{\circ}{\mathbb{P}}$ be the mixed curvature of the
Berwald derivative. If the horizontal projector $\mathcal{H}$ is
homogeneous, then the mixed curvature of the Berwald derivative induced by $%
\mathcal{H}$ satisfies
\begin{align*}
\overset{\circ}{\mathbb{P}}(X,Y)U=\overset{\circ}{\mathbb{P}}(U,X)Y=0, \ \
\forall X,Y\in\Gamma((\rho,\eta)TE,(\rho,\eta)\tau _{E},E),
\end{align*}
\end{proposition}
\begin{proof}
Using (\ref{reclaim2}) and the components of the Berwald derivative, yield
\begin{align*}
\overset{\circ}{\mathbb{P}}(X,Y)U&=\dot{\tilde{X}}^{d}Y^{c}U^b\left(
\rho ,\eta
,h\right) \mathbb{P}_{~b~cd}^{a}\tilde{\delta}_{a}+\dot{\tilde{X}}^{d}Y^{c}U^b\left( \rho ,\eta ,h\right) \mathbb{\tilde{P}}_{~b~cd}^{a}\dot{\tilde{%
\partial}}_{a}\\
&=\dot{\tilde{X}}^{d}Y^{c}U^b\left( \rho ,\eta
,h\right)\Gamma(\tilde{\rho}, \text{Id}_E)(\dot{\tilde{\partial}}_d)(\dot{\partial}_b(\rho, \eta)\Gamma^a_c)(\tilde{\delta}_{a}+\dot{\tilde{%
\partial}}_{a})\\
&=\dot{\tilde{X}}^{d}Y^{c}U^b\dot{\partial}_d\dot{\partial}_b((\rho, \eta)\Gamma^a_c)(\tilde{\delta}_{a}+\dot{\tilde{%
\partial}}_{a}).
\end{align*}
But (\ref{esi1}) obtains $U^b\dot{\partial}_d\dot{\partial}_b((\rho, \eta)\Gamma^a_c)=0$. Therefore the above equation gives $\overset{\circ}{\mathbb{P}}(X,Y)U=0$. Similarly we can deduce the second equation.
\end{proof}
\begin{defn}
The modules endomorphism $\mathcal{H}_S$ on $\Gamma \left( \left( \rho ,\eta \right) TE,\left( \rho ,\eta \right) \tau
_{E},E\right)$ defined by
$$\mathcal{H}_S(X):=\frac{1}{2}\Big(X+[\mathcal{J}_{(g,h)}X, S]_{( \rho ,\eta ) TE}-\mathcal{J}_{(g,h)}[X, S]_{( \rho ,\eta ) TE} \Big), $$
is a horizontal projector which is called the horizontal projector associated to the $(\rho,\eta)$-semispray $S$.
\end{defn}
Using the above definition, we get
\begin{align*}
\mathcal{H}_S(\dot{\tilde{\partial}}_{a})=0,\ \ \mathcal{H}_S({\tilde{\partial}}_{a})&={\tilde{\partial}}_{a}+\mathcal{H}^b_a\dot{\tilde{\partial}}_b,
\end{align*}
where
\begin{align}\label{esi2}
\mathcal{H}^b_a:&=\frac{1}{2}[U^c((g^d_c\tilde{g}^b_eL^e_{da})\circ h\circ\pi)-U^c(\rho^i_a\circ h\circ\pi)(\partial_i(g^e_c\circ h\circ\pi))(\tilde{g}^b_e\circ h\circ\pi)\nonumber\\
&\ \ \ -(g^c_d\circ h\circ\pi)U^d(\rho^i_c\circ h\circ\pi)(\partial_i(\tilde{g}^b_a\circ h\circ\pi))\nonumber\\
&\ \ \ -2(\tilde{g}^c_a\circ h\circ\pi)\dot{\partial}_c(G^b-\frac{1}{4}F^b)].
\end{align}
\begin{lemma}\label{Lemma7.7}
If $S$ is a $(\rho, \eta)$-spray, then $\mathcal{H}_S$ is homogenous.
\end{lemma}
\begin{proof}
Using Lemma \ref{Lemma7.5}, it is sufficient to show that
$U^f\dot{\partial}_f\mathcal{H}^b_a=\mathcal{H}^b_a$. Now 
(\ref{esi2}) infers
\begin{align*}
U^f\dot{\partial}_f\mathcal{H}^b_a:&=\frac{1}{2}[U^f((g^d_f\tilde{g}^b_eL^e_{da})\circ h\circ\pi)-U^f(\rho^i_a\circ h\circ\pi)(\partial_i(g^e_f\circ h\circ\pi))
(\tilde{g}^b_e\circ h\circ\pi)\nonumber\\
&\ \ \ -(g^c_f\circ h\circ\pi)U^f(\rho^i_c\circ
h\circ\pi)(\partial_i(\tilde{g}^b_a\circ
h\circ\pi))\nonumber\\
&\ \ \ -2(\tilde{g}^c_a\circ
h\circ\pi)U^f\dot{\partial}_f\dot{\partial}_c(G^b-\frac{1}{4}F^b)].
\end{align*}
Comparing the above equation with (\ref{esi2}), one can see that all the
terms of $U^f\dot{\partial}_f\mathcal{H}^b_a$ and $\mathcal{H}^b_a$
are identical except the last term. But since $S$ is a $(\rho,
\eta)$-spray, then using (\ref{IM}) the last term reduces to
$$U^f\dot{\partial}_f\dot{\partial}_c(G^b-\frac{1}{4}F^b)=\dot{\partial}_c(G^b-\frac{1}{4}F^b).$$
Thus the last terms of $U^f\dot{\partial}_f\mathcal{H}^b_a$ and
$\mathcal{H}^b_a$ are identical, too. Therefore
$U^f\dot{\partial}_f\mathcal{H}^b_a=\mathcal{H}^b_a$.
\end{proof}
\begin{lemma}\label{Important}
If $S$ is a $(\rho, \eta)$-spray and $(\rho, \eta)\nabla$ is the
Berwald derivative induced by $\mathcal{H}_S$, then $(\rho,
\eta)\nabla_SU=0$.
\end{lemma}
\begin{proof}
We have
\begin{align*}
(\rho, \eta)\nabla_SU=\tilde{\rho}(S)(U^b)(\tilde{\delta}_b+\dot{\tilde{\partial}}_b)+U^b\Big((\rho, \eta)\nabla_S\tilde{\delta}_b+(\rho, \eta)\nabla_S\dot{\tilde{\partial}}_b\Big).
\end{align*}
Since $(\rho, \eta)\nabla_{\dot{\tilde{\partial}}_a}\tilde{\delta}_b=(\rho, \eta)\nabla_{\dot{\tilde{\partial}}_a}\dot{\tilde{\partial}}_b=0$, and
\begin{align*}
(\rho, \eta)\nabla_{\tilde{\delta}_a}\tilde{\delta}_b&=\dot{\partial}_b((\rho, \eta)\Gamma^c_a)\tilde{\delta}_c=-\dot{\partial}_b\mathcal{H}^c_a\tilde{\delta}_c,\\
(\rho, \eta)\nabla_{\tilde{\delta}_a}\dot{\tilde{\partial}}_b&=\dot{\partial}_b((\rho, \eta)\Gamma^c_a)\dot{\tilde{\partial}}_c=-\dot{\partial}_b\mathcal{H}^c_a\dot{\tilde{\partial}}_c,
\end{align*}
then we obtain
\[
(\rho, \eta)\nabla_SU=\mathcal{A}^b(\tilde{\delta}_b+\dot{\tilde{\partial}}_b),
\]
where
\begin{align}\label{esi0}
\mathcal{A}^b&=-2(G^b-\frac{1}{4}F^b)-\frac{1}{2}U^mU^r(g^a_r\circ h\circ\pi)\Big((g^d_m\circ h\circ\pi)(\tilde{g}^b_e\circ h\circ\pi)(L^e_{da}\circ h\circ\pi)\nonumber\\
&\ \ \ -(\rho^i_a\circ h\circ\pi)(\tilde{g}^b_e\circ h\circ\pi)\partial_i(g^e_m\circ h\circ\pi)-2(\tilde{g}^c_a\circ h\circ\pi)\dot{\partial}_m\dot{\partial}_c(G^b-\frac{1}{4}F^b)\nonumber\\
&\ \ \ -(g^c_m\circ h\circ\pi)(\rho^i_c\circ h\circ\pi)\partial_i(\tilde{g}^b_a\circ h\circ\pi)\Big).
\end{align}
But we have
\begin{align*}
U^mU^r((g^a_rg^d_m\tilde{g}^b_eL^e_{da})\circ h\circ\pi)=-U^mU^r((g^a_rg^d_m\tilde{g}^b_eL^e_{ad})\circ h\circ\pi).
\end{align*}
Replacing $r\leftrightarrow m$ and $d\leftrightarrow a$ in the right side of the above equation implies that
\begin{align*}
U^mU^r((g^a_rg^d_m\tilde{g}^b_eL^e_{da})\circ h\circ\pi)=-U^rU^m((g^d_mg^a_r\tilde{g}^b_eL^e_{da})\circ h\circ\pi),
\end{align*}
attains
\begin{align}\label{esi00}
U^mU^r((g^a_rg^d_m\tilde{g}^b_eL^e_{da})\circ h\circ\pi)=0.
\end{align}
Since $g^a_r\tilde{g}^b_a=\delta^b_r$, then we deduce that
\[
U^mU^r((g^a_rg^c_m\rho^i_c)\circ h\circ\pi)\partial_i(\tilde{g}^b_a\circ h\circ\pi)=-U^mU^r((g^c_m\tilde{g}^b_a\rho^i_c)\circ h\circ\pi)\partial_i(g^a_r\circ h\circ\pi).
\]
Replacing $a\rightarrow e$, $c\rightarrow a$ and $r\leftrightarrow m$ in the right side of the above equation yields
\begin{align}\label{esi000}
U^mU^r((g^a_rg^c_m\rho^i_c)\circ h\circ\pi)\partial_i(\tilde{g}^b_a\circ h\circ\pi)=-U^rU^m((g^a_r\tilde{g}^b_e\rho^i_a)\circ h\circ\pi)\partial_i(g^e_m\circ h\circ\pi).
\end{align}
We have
\begin{align}\label{esi4}
U^mU^r(g^a_r\circ h\circ\pi)(\tilde{g}^c_a\circ h\circ\pi)\dot{\partial}_m\dot{\partial}_c(G^b-\frac{1}{4}F^b)=U^mU^c\dot{\partial}_m\dot{\partial}_c(G^b-\frac{1}{4}F^b).
\end{align}
Since $S$ is $(\rho, \eta)$-spray, then we have (\ref{IM}). Differentiating (\ref{IM}) with respect to the $y^c$ and then multiplying $U^c$ in it, yield
\[
U^cU^b\dot{\partial}_c\dot{\partial}_b(G^a-\frac{1}{4}F^a)=U^c\dot{\partial}_c(G^a-\frac{1}{4}F^a)=2(G^a-\frac{1}{4}F^a).
\]
Setting the above equation in (\ref{esi4}), results
\begin{align}\label{esi0000}
U^mU^r(g^a_r\circ h\circ\pi)(\tilde{g}^c_a\circ h\circ\pi)\dot{\partial}_m\dot{\partial}_c(G^b-\frac{1}{4}F^b)=2(G^b-\frac{1}{4}F^b).
\end{align}
Setting (\ref{esi00}), (\ref{esi000}) and (\ref{esi0000}) in (\ref{esi0}), yields $\mathcal{A}^b=0$ and consequently $(\rho, \eta)\nabla_SU=0$.
\end{proof}
\begin{lemma}
Let $f$ be a function of class $C^1$ on $E$ and $S$ be the canonical $\left( \rho ,\eta \right) $-semispray associated to
the mechanical $\left( \rho ,\eta \right) $-system $\left( \left( E,\pi
,M\right) ,F_{e},\left( \rho ,\eta \right) \Gamma \right) $ and locally invertible vector bundles morphism $\left( g,h\right) $. Then $\bar{S}=S+f\mathbb{C}$ is a $\left( \rho ,\eta \right) $-spray if and only if $f$ is 1-homogenous.
\end{lemma}
\begin{proof}
Using the locally expressions of $S$ and $\mathbb{C}$, they yield
\begin{align*}
\bar{S}=(g^a_b\circ
h\circ\pi)U^b\tilde{\partial}_a+(fU^a-2(G^a-\frac{1}{4}F^a))\dot{\tilde{\partial}}_a.
\end{align*}
Thus $\bar{S}$ is $\left( \rho ,\eta \right) $-spray if and only if
\begin{align}\label{P1}
U^b\dot{\partial}_b(fU^a-2(G^a-\frac{1}{4}F^a))=2(fU^a-2(G^a-\frac{1}{4}F^a)).
\end{align}
Since $S$ is $\left( \rho ,\eta \right) $-spray, then using (\ref{IM}), it concludes
\begin{align*}
U^b\dot{\partial}_b(fU^a-2(G^a-\frac{1}{4}F^a))&=U^aU^b\dot{\partial}_bf+fU^a-4(G^a-\frac{1}{4}F^a)\\
&=U^aU^b\dot{\partial}_bf-fU^a+2(fU^a-2(G^a-\frac{1}{4}F^a)).
\end{align*}
Two above equations, have an outcome that $\bar{S}$ is $\left( \rho ,\eta \right) $-spray if and only if
\begin{align*}
U^aU^b\dot{\partial}_bf=fU^a.
\end{align*}
But the above equation is equaivalent to $((\rho, \eta)\nabla^v_Uf)\mathbb{C}=f\mathbb{C}$. Since $\mathbb{C}$ is a nonzero section, then $\bar{S}$ is $\left( \rho ,\eta \right) $-spray if and only if $(\rho, \eta)\nabla^v_Uf=f$.
\end{proof}
\begin{defn}
Two $(\rho, \eta)$-sprays $S$ and $\bar{S}$ are called projectively
related if there exists a $1$-homogenous  function $f$ of class $C^1$
on $E$ such that $\bar{S}=S+f\mathbb{C}$. In this case we say that
$\bar{S}$ is a projective change of $S$.
\end{defn}
\begin{theorem}
If $\bar{S}=S+f\mathbb{C}$ is a projectively change of the $(\rho, \eta)$-spray $S$, then the horizontal projectors associated to $S$ and $\bar{S}$ are related by
\begin{equation}\label{IM1}
\mathcal{H}_{\bar{S}}=\mathcal{H}_S+\frac{1}{2}(f\mathcal{J}_{(g,
h)}+(\rho, \eta)\nabla^v_{\mathcal{J}_{(g, h)}}f\otimes\mathbb{C}).
\end{equation}
\end{theorem}
\begin{proof}
Let $X=X^a\tilde{\partial}_a+\tilde{X}^a\dot{\tilde{\partial}}_a$.
Thus
\begin{align}\label{p2}
\mathcal{H}_{\bar
S}(X)=X^a\mathcal{H}_{\bar{S}}(\tilde{\partial}_a)=X^a\tilde{\partial}_a+X^a\bar{\mathcal{H}}^b_a\dot{\tilde{\partial}}_b,
\end{align}
where
\begin{align*}
\bar{\mathcal{H}}^b_a&=\frac{1}{2}[U^c((g^d_c\tilde{g}^b_eL^e_{da})\circ h\circ\pi)-U^c(\rho^i_a\circ h\circ\pi)(\partial_i(g^e_c\circ h\circ\pi))(\tilde{g}^b_e\circ h\circ\pi)\nonumber\\
&\ \ \ -(g^c_d\circ h\circ\pi)U^d(\rho^i_c\circ h\circ\pi)(\partial_i(\tilde{g}^b_a\circ h\circ\pi))\\
&\ \ \ +(\tilde{g}^c_a\circ h\circ\pi)\dot{\partial}_c(fU^b-2(G^b-\frac{1}{4}F^b))].
\end{align*}
This equation together with (\ref{esi2}), give
\begin{align*}
\bar{\mathcal{H}}^b_a=\mathcal{H}^b_a+\frac{1}{2}(\tilde{g}^c_a\circ h\circ\pi)\dot{\partial}_c(fU^b)=\mathcal{H}^b_a
+\frac{1}{2}(\tilde{g}^c_a\circ h\circ\pi)U^b\dot{\partial}_cf+\frac{1}{2}f(\tilde{g}^b_a\circ h\circ\pi).
\end{align*}
Setting this equation in (\ref{p2}), we get
\begin{align*}
\mathcal{H}_{\bar S}(X)&=\mathcal{H}_S(X)+\frac{1}{2}X^a((\tilde{g}^c_a\circ h\circ\pi)U^b\dot{\partial}_cf+f(\tilde{g}^b_a\circ h\circ\pi))\dot{\tilde{\partial}}_b\\
&=\mathcal{H}_S(X)+\frac{1}{2}(f\mathcal{J}_{(g, h)}(X)+(\rho,
\eta)\nabla^v_{\mathcal{J}_{(g, h)}(X)}f\mathbb{C}).
\end{align*}
\end{proof}
\begin{theorem}
If $\bar{S}=S+f\mathbb{C}$ is a projectively change of the $(\rho,
\eta)$-spray $S$, then the mixed curvature of the Berwald
derivatives induced by $\mathcal{H}_S$ and $\mathcal{H}_{\bar S}$
are related by
\begin{align*}
\overline{\overset{\circ}{\mathbb{P}}}(X,
Y)Z&=\overset{\circ}{\mathbb{P}}(X,
Y)Z-(\dot{\tilde{X}}^dY^cZ^b\dot{\partial}_d\dot{\partial}_bA^a_c+X^fY^cA^d_fZ^b\tilde{g}^e_c\dot{\partial}_d\dot{\partial}_b\dot{\partial}_e(G^a-\frac{1}{4}F^a))
\tilde{\delta}_a\\
&\ \ \
+(-\dot{\tilde{X}}^dY^cZ^b(A^a_e\dot{\partial}_d\dot{\partial}_bA^e_c-\tilde{g}^f_c\dot{\partial}_d\dot{\partial}_b\dot{\partial}_f(G^e-\frac{1}{4}F^e)A^a_e)\\
&\ \ \
-\dot{\tilde{X}}^dY^c\tilde{Z}^b\dot{\partial}_d\dot{\partial}_bA^a_c
-X^fY^cA^d_f(Z^tA^b_t+\dot{\tilde{Z}}^b)\tilde{g}^e_c\dot{\partial}_d\dot{\partial}_b\dot{\partial}_e(G^a-\frac{1}{4}F^a)\\
&\ \ \
-\dot{\tilde{X}}^dY^cZ^fA^b_f\tilde{g}^e_c\dot{\partial}_d\dot{\partial}_b\dot{\partial}_e(G^a-\frac{1}{4}F^a))\dot{\tilde{\partial}}_a,
\end{align*}
where
$X=X^d\bar{\tilde{\delta}}_d+\dot{\tilde{X}}^d\dot{\partial}_d$ and
$Y=Y^c\bar{\tilde{\delta}}_c+\dot{\tilde{Y}}^c\dot{\partial}_c$,
$Z=Z^b\bar{\tilde{\delta}}_b+\dot{\tilde{Z}}^b\dot{\partial}_b$ and
\[
A^b_a=\frac{1}{2}(\tilde{g}^c_a\circ
h\circ\pi)U^b\dot{\partial}_cf+\frac{1}{2}f(\tilde{g}^b_a\circ
h\circ\pi).
\]
\end{theorem}
\begin{proof}
Since
$X=X^a\bar{\tilde{\delta}}_a+\dot{\tilde{X}}^a\dot{\partial}_a$  and
$\bar{\tilde{\delta}}_a=\tilde{{\delta}}_a+A^c_a\dot{{\partial}}_a$, 
then $X$ has the locally expression
$X=X^a\tilde{{\delta}}_a+(X^cA^a_c+\dot{\tilde{X}}^a)\dot{\partial}_a$  with respect to the $\{\bar{\tilde{\delta}}_a,
\dot{\tilde{\partial}}_a\}$.
Similarly, 
$Y=Y^a\tilde{{\delta}}_a+(Y^cA^a_c+\dot{\tilde{Y}}^a)\dot{\partial}_a$
and
$Z=Z^a\tilde{{\delta}}_a+(Z^cA^a_c+\dot{\tilde{Z}}^a)\dot{\partial}_a$.
Using these and the definition of the mixed curvature of $(\rho,
\eta)\nabla$, we get
\begin{align}\label{lili}
\overset{\circ}{\mathbb{P}}(X, Y)Z&=(\rho, \eta,
h)\mathbb{R}(\mathcal{V}_{ S}X,
\mathcal{H}_{S}Y)Z\nonumber\\
&=Y^c(X^fA^d_f+\dot{\tilde{X}}^d)(\rho, \eta,
h)\mathbb{R}(\dot{\tilde{\partial}}_d,
\tilde{\delta}_c)(Z^b\tilde{\delta}_b+(Z^eA^b_e+\dot{\tilde{Z}}^b)\tilde{\dot{\partial}}_b)\nonumber\\
&=
-Y^c(X^fA^d_f+\dot{\tilde{X}}^d)Z^b\dot{\partial}_d\dot{\partial}_bH^a_c\tilde{\delta}_a\nonumber\\
&\ \ \
-Y^c(X^fA^d_f+\dot{\tilde{X}}^d)(Z^eA^b_e+\dot{\tilde{Z}}^b)\dot{\partial}_d\dot{\partial}_bH^a_c\tilde{\dot{\partial}}_a.
\end{align}
Now we compute the mixed curvature of $(\rho, \eta)\bar{\nabla}$ as
follows
\begin{align*}
\overline{\overset{\circ}{\mathbb{P}}}(X, Y)Z&=(\rho, \eta,
h)\overline{\mathbb{R}}(\mathcal{V}_{\bar{S}}X,
\mathcal{H}_{\bar{S}}Y)Z=Y^c\dot{\tilde{X}}^d(\rho, \eta,
h)\overline{\mathbb{R}}(\dot{\tilde{\partial}}_d,
\bar{\tilde{\delta}}_c)(Z^b\bar{\tilde{\delta}}_b+\dot{\tilde{Z}}^b\tilde{\dot{\partial}}_b)\nonumber\\
&\ \ \
=-Y^c\dot{\tilde{X}}^dZ^b\dot{\partial}_d\dot{\partial}_b\bar{H}^a_c\bar{\tilde{\delta}}_a
-Y^c\dot{\tilde{X}}^d\dot{\tilde{Z}}^b\dot{\partial}_d\dot{\partial}_b\bar{H}^a_c\dot{\tilde{\partial}}_a.
\end{align*}
Setting $\bar{H}^a_c=H^a_c+A^a_c$ and
$\bar{\tilde{\delta}}_a=\tilde{{\delta}}_a+A^c_a\dot{{\partial}}_a$
in the above equation, yields
\begin{align}\label{lili1}
\overline{\overset{\circ}{\mathbb{P}}}(X,
Y)Z&=(-\dot{\tilde{X}}^dY^cZ^b(\dot{\partial}_d\dot{\partial}_bA^a_c+\dot{\partial}_d\dot{\partial}_bH^a_c))\tilde{{\delta}}_a
+(-\dot{\tilde{X}}^dY^cZ^bA^a_e(\dot{\partial}_d\dot{\partial}_bA^e_c\nonumber\\
&\ \ \
+\dot{\partial}_d\dot{\partial}_bH^e_c)-\dot{\tilde{X}}^dY^c\dot{\tilde{Z}}^b(\dot{\partial}_d\dot{\partial}_bA^a_c
+\dot{\partial}_d\dot{\partial}_bH^a_c))\dot{\tilde{\partial}}_a.
\end{align}
But from (\ref{esi2}) it is easy to see that
\[
\dot{\partial}_d\dot{\partial}_bH^a_c=-(\tilde{g}^f_c\circ
h\circ\pi)\dot{\partial}_d\dot{\partial}_b\dot{\partial}_f(G^a-\frac{1}{4}F^a).
\]
Putting the above equation in (\ref{lili}) and (\ref{lili1}) and
then comparing the new equations we obtain the assertion.
\end{proof}
\begin{theorem}
Let $S$ and $\bar{S}$ be two $(\rho, \eta)$-sprays of the mechanical $\left( \rho ,\eta \right)
$-systems $\left( \left( E,\pi ,M\right), F_{e}, \left( \rho ,\eta
\right) \Gamma \right) $ and $\left( \left( E,\pi ,M\right), \bar{F}_{e}, \left( \rho ,\eta
\right) \bar{\Gamma}\right) $, respectively. Then $S$ and
$\bar{S}$ are projectively related if and only if the geodesics of these systems with respect to the $(\rho, \eta)$-sprays
$S$ and $\bar{S}$ differ only in a strictly increasing parameter
transformation.
\end{theorem}
\begin{proof}
First, we assume that $\bar{S}=S+f\mathbb{C}$ such that $f\left(
\lambda u_{x}\right) =\lambda f\left( u_{x}\right)$. Let $c$ be a
geodesic of the mechanical $\left( \rho ,\eta \right) $-system
$\left( \left( E,\pi ,M\right) ,F_{e},\left( \rho ,\eta \right)
\Gamma \right) $ with respect to the $(\rho, \eta)$-spray $S$. Then
\begin{equation}
\frac{dy^{a}(t) }{dt}+2( G^{a}-\frac{1}{4}F^{a})( \dot{c}(t)) =0.
\end{equation}
Now, we can extract the parameter $s$ by solving the ODE%
\begin{equation}
\frac{d^{2}s\left( t\right) }{dt^{2}}=-f\left( \dot{c}\left(
t\right) \right) \frac{ds\left( t\right) }{dt},~\frac{ds\left(
t\right) }{dt}>0.
\end{equation}
With this parametrization, we have
\begin{equation}
\dot{c}\left( t\right) =\dot{c}\left( s\left( t\right) \right) \frac{%
ds\left( t\right) }{dt},
\end{equation}
which gives
\begin{equation}
y^{a}\left( t\right) =y^{a}\left( s\left( t\right) \right)
\frac{ds\left( t\right) }{dt}.
\end{equation}
The above equation implies that
\begin{eqnarray}
\frac{dy^{a}(t) }{dt} &=&\frac{d}{dt}(y^{a}(s)
\frac{ds}{dt})  \\
&=&\frac{dy^{a}(s) }{dt}\frac{ds}{dt}+y^{a}(s)
\frac{d^{2}s}{dt^{2}}  \notag \\
&=&\frac{dy^{a}(s) }{ds}( \frac{ds}{dt}) ^{2}+y^{a}(s)
\frac{d^{2}s}{dt^{2}}.  \notag
\end{eqnarray}
Hence
\begin{eqnarray}\label{Khobe}
\frac{dy^{a}(s) }{ds} &=&\frac{-2(G^{a}-\frac{1}{4}%
F^{a})(\dot{c}(t))\frac{ds%
}{dt}+f(\dot{c}(t)) \frac{ds}{dt}y^{a}(t) }{(\frac{ds}{dt}) ^{3}}\nonumber\\
&=&\frac{-2(\bar{G}^{a}-\frac{1}{4}\bar{F}^{a})( \dot{c}(t) )}{(
\frac{ds}{dt}) ^{2}}
\notag \\
&=&-2(\bar{G}^{a}-\frac{1}{4}\bar{F}^{a})(\dot{c}(t))(\frac{dt}{ds})
^{2}
\notag \\
&=&-2(\bar{G}^{a}-\frac{1}{4}\bar{F}^{a})(\dot{c}(
t))(\frac{dt}{ds}) ^{2}.
\end{eqnarray}
Since $\bar{S}$ is an spray, then for any $a=1,\cdots, r$,
$\bar{G}^{a}-\frac{1}{4}\bar{F}^{a}$ is a homogenous function of
degree 2. Thus
\[
(\bar{G}^{a}-\frac{1}{2}\bar{F}^{a})( \dot{c}(t))(\frac{dt}{ds})
^{2}=(\bar{G}^{a}-\frac{1}{4}\bar{F}^{a})(\dot{c}(t)(\frac{dt}{ds}))=(\bar{G}^{a}-\frac{1}{4}\bar{F}^{a})(\dot{c}(s)).
\]
Setting this equation in (\ref{Khobe}), yields
\begin{equation}
\frac{dy^{a}(s)}{ds}=-2(\bar{G}^{a}-\frac{1}{4}\bar{F}^{a})(\dot{c}(s)).
\end{equation}
Conversely, let $S$ and $\overline{S}$ have the same geodesic $c$ around any point. To show that $S$ and $\overline{S}$ are projectively related, it is sufficient to prove 
\begin{align}\label{Khob1}
2(G^a-\frac{1}{4}F^a)(u_x)-2(\bar{G}^a-\frac{1}{4}\bar{F}^a)(u_x)=f(u_x)U^a(u_x),
\end{align}
where $f$ is a 1-homogenous function on $E$ and $u_x\in E_x$. Now, assume that $t$ and $s$ be the parameters of the geodesic $c$ with respect to $S$ and $\overline{S}$. Then
\begin{align}
\frac{dy^{a}(t)}{dt}+2(G^{a}-\frac{1}{4}F^{a})(\dot{c}(t))
&=0,\label{geod1}\\
\frac{dy^{a}(s)}{dt}+2(\bar{G}^{a}-\frac{1}{4}\bar{F}^{a})(\dot{c}(s))
&=0.\label{geod2}
\end{align}
Now we parametrize $c$ by a common parameter $u$. Then
\begin{align*}
y^a(t)=y^a(u)\frac{du}{dt},\ \ \ y^a(s)=y^a(u)\frac{du}{ds}.
\end{align*}
These equations give
\begin{align*}
\frac{d}{dt}y^a(t)&=\frac{d}{dt}(y^a(u)\frac{du}{dt})=y^a(u)\frac{d^2u}{dt^2}+(\frac{du}{dt})^2\frac{d}{du}y^a(u),\\
\frac{d}{ds}y^a(s)&=\frac{d}{ds}(y^a(u)\frac{du}{ds})=y^a(u)\frac{d^2u}{ds^2}+(\frac{du}{ds})^2\frac{d}{du}y^a(u).
\end{align*}
Setting these equation in (\ref{geod1}) and (\ref{geod2}) and using the 2-homogeniuty of $(G^a-\frac{1}{4}F^a)$, yield
\begin{align*}
y^a(u)\frac{d^2u}{dt^2}+(\frac{du}{dt})^2\frac{d}{du}y^a(u)+2(G^a-\frac{1}{4}F^a)(\dot{c}(u))(\frac{du}{dt})^2&=0,\\
y^a(u)\frac{d^2u}{ds^2}+(\frac{du}{ds})^2\frac{d}{du}y^a(u)+2(\bar{G}^a-\frac{1}{4}\bar{F}^a)(\dot{c}(u))(\frac{du}{ds})^2&=0.
\end{align*}
The above equations, yield
\[
2(G^a-\frac{1}{4}F^a)(\dot{c}(u))-2(\bar{G}^a-\frac{1}{4}\bar{F}^a)(\dot{c}(u))=(\frac{\frac{d^2u}{dt^2}}{(\frac{du}{dt})^2}-\frac{\frac{d^2u}{ds^2}}{(\frac{du}{ds})^2})y^a(u).
\]
Since the above equation holds for any geodesic, it results that (\ref{Khob1}) holds for some 1-homogenous  function $f$ on $E$.
\end{proof}
\begin{theorem}
Let $S$ and $\bar{S}$ be two $(\rho, \eta)$-sprays of the mechanical $\left( \rho ,\eta \right)
$-systems $\left( \left( E,\pi ,M\right), F_{e}, \left( \rho ,\eta
\right) \Gamma \right) $ and $\left( \left( E,\pi ,M\right), \bar{F}_{e}, \left( \rho ,\eta
\right) \bar{\Gamma}\right) $, respectively. Then $\bar{S}$
and $S$ are projectively related if and only if the Berwald
derivatives $\nabla$ and $\bar{\nabla}$ induced by $\mathcal{H}_{S}$
and $\mathcal{H}_{\bar{S}}$ respectively, are related by
\begin{align}\label{Ex}
(\rho, \eta)\bar{\nabla}_XY&=(\rho,
\eta)\nabla_XY-X^aY^b\dot{\partial}_bA^c_a\tilde{\delta}_c+\Big(-X^aY^bA^c_m\dot{\partial}_bH^m_a-X^aY^bA^c_m\dot{\partial}_bA^m_a\nonumber\\
&\ \ \
-X^a\dot{\tilde{Y}}^b\dot{\partial}_bA^c_a-X^aA^m_aY^d\dot{\partial}_mA^c_d-X^aY^d\rho^i_a\partial_iA^c_d-X^aY^dH^m_a\dot{\partial}_mA^c_d\nonumber\\
&\ \ \
+X^aY^dA^b_d\dot{\partial}_bH^c_a-\dot{\tilde{X}}^aY^d\dot{\partial}_aA^c_d\Big)\dot{\tilde{\partial}}_c,
\end{align}
where
$X=X^a\bar{\tilde{\delta}}_a+\dot{\tilde{X}}^a\dot{\partial}_a$ and
$Y=Y^b\bar{\tilde{\delta}}_b+\dot{\tilde{Y}}^b\dot{\partial}_b$.
\end{theorem}
\begin{proof}
Let $\bar{S}$ and $S$ be projectively related $(\rho, \eta)$-sprays and $\mathcal{H}_{\bar{S}}$ and $\mathcal{H}_S$ be the horizontal projectors induced by $\bar{S}$ and $S$, respectively. Then 
\begin{align}\label{del0}
\bar{{\mathcal{H}}}^b_a=\mathcal{H}^b_a+A^b_a.
\end{align}
Denoting by $\{\bar{\tilde{\delta}}_a, \dot{\tilde{\partial}}_a\}$ and $\{\tilde{\delta}_a, \dot{\tilde{\partial}}_a\}$ the adapted basis associated with $\mathcal{H}_{\bar{S}}$ and $\mathcal{H}_S$, respectively, then using (\ref{del0}) we have
\begin{align}\label{del}
\bar{\tilde{\delta}}_a=\tilde{\partial}_a+\bar{\mathcal{H}}_a^b\dot{\tilde{\partial}}_b
=\tilde{\partial}_a+(\mathcal{H}_a^b+A_a^b)\dot{\tilde{\partial}}_b=\tilde{\delta}_a+A^b_a\dot{\tilde{\partial}}_b.
\end{align}
Using the definition of the Berwald derivative, (\ref{del0}) and (\ref{del}), it attains
\begin{align}\label{del1}
(\rho, \eta)\bar{\nabla}_{\bar{\tilde{\delta}}_a}\bar{\tilde{\delta}}_b&=-(\dot{\partial}_b\bar{\mathcal{H}}^c_a)
\bar{\tilde{\delta}}_c
=-(\dot{\partial}_b\mathcal{H}^c_a+\dot{\partial}_bA^c_a)(\tilde{\delta}_c+A^d_c\dot{\tilde{\partial}}_d)\nonumber\\
&=-(\dot{\partial}_b\mathcal{H}^c_a+\dot{\partial}_bA^c_a)\bar{\tilde{\delta}}_c-A^c_d(\dot{\partial}_b\mathcal{H}^d_a+\dot{\partial}_bA^d_a)\dot{\tilde{\partial}}_c.
\end{align}
Similarly
\begin{align}\label{del2}
(\rho, \eta)\bar{\nabla}_{\bar{\tilde{\delta}}_a}\dot{\tilde{\partial}}_b=-(\dot{\partial}_b\mathcal{H}^c_a+\dot{\partial}_bA^c_a)\dot{\tilde{\partial}}_c.
\end{align}
Also for the Berwald derivative $(\rho, \eta)\bar{\nabla}$, we have
\begin{align}\label{del3}
(\rho, \eta)\bar{\nabla}_{\dot{\tilde{\partial}}_a}\dot{\tilde{\partial}}_b=(\rho, \eta)\bar{\nabla}_{\dot{\tilde{\partial}}_a}\bar{\tilde{\delta}}_b=0.
\end{align}
Now, let
$X=X^a\bar{\tilde{\delta}}_a+\dot{\tilde{X}}^a\dot{\partial}_a$ and
$Y=Y^b\bar{\tilde{\delta}}_b+\dot{\tilde{Y}}^b\dot{\partial}_b$.
Then (\ref{del1}), (\ref{del2}) and (\ref{del3}) give
\begin{align}\label{del4}
(\rho, \eta)\bar{\nabla}_{X}Y&=\Big(-X^aY^b(\dot{\partial}_b\mathcal{H}^c_a+\dot{\partial}_bA^c_a)+X^a\rho^i_a\partial_i(Y^c)+X^a(H^d_a+A^d_a)
\dot{\partial}_d(Y^c)\nonumber\\
&\ \ \ +\dot{\tilde{X}}^a\dot{\partial}_a(Y^c)\Big)\tilde{\delta}_c+\Big(-X^aY^bA^c_d(\dot{\partial}_b\mathcal{H}^d_a+\dot{\partial}_bA^d_a)+X^a\rho^i_a\partial_i(Y^b)A^c_b\nonumber\\
&\ \ \ +X^a(H^d_a+A^d_a)
\dot{\partial}_d(Y^b)A^c_b-X^a\dot{\tilde{Y}}^b(\dot{\partial}_b\mathcal{H}^c_a+\dot{\partial}_bA^c_a)+X^a\rho^i_a\partial_i(\tilde{Y}^c)\nonumber\\
&\ \ \ +X^a(H^d_a+A^d_a)
\dot{\partial}_d(\dot{\tilde{Y}}^c)+\dot{\tilde{X}}^a\dot{\partial}_a(\dot{\tilde{Y}}^c)
+\dot{\tilde{X}}^a\dot{\partial}_a(Y^b)A^c_b\Big)\dot{\tilde{\partial}}_c.
\end{align}
Now we consider the locally expressions of sections $X$ and $Y$ with respect to the adapted basis $\{\tilde{\delta}_a, \dot{\tilde{\partial}}_a\}$, i.e.,
\[
X=X^a\tilde{\delta}_a+(\dot{\tilde{X}}^a+X^cA^a_c)\dot{\tilde{\partial}}_a,\
\
Y=Y^b\tilde{\delta}_b+(\dot{\tilde{Y}}^b+Y^dA^b_d)\dot{\tilde{\partial}}_b,
\]
and compute $(\rho, \eta)\nabla_{X}Y$. Direct calculations yield
\begin{align}\label{del5}
(\rho, \eta)\nabla_{X}Y&=\Big(-X^aY^b\dot{\partial}_b\mathcal{H}^c_a+X^a\rho^i_a\partial_i(Y^c)+X^a\mathcal{H}^d_a\dot{\partial}_d(Y^c)
+\dot{\tilde{X}}^a\dot{\partial}_a(Y^c)\nonumber\\
&\ \ \
+X^bA^a_b\dot{\partial}_a(Y^c)\Big)\tilde{\delta}_c+\Big(-X^a\dot{\tilde{Y}}^b\dot{\partial}_b\mathcal{H}^c_a+X^a\rho^i_a\partial_i(\dot{\tilde{Y}}^c)
+X^a\mathcal{H}^d_a\dot{\partial}_d(\dot{\tilde{Y}}^c)\nonumber\\
&\ \ \
+\dot{\tilde{X}}^a\dot{\partial}_a(\dot{\tilde{Y}}^c)+X^bA^a_b\dot{\partial}_a(\dot{\tilde{Y}}^c)+X^bA^a_b\dot{\partial}_a(Y^d)A^c_d
+X^bA^a_bY^d\dot{\partial}_aA^c_d
\nonumber\\
&\ \ \ +X^a\rho^i_a\partial_i(Y^d)A^c_d+X^a\mathcal{H}^d_a\dot{\partial}_d(Y^b)A^c_b+X^a\rho^i_a\partial_i(A^c_b)Y^b
\nonumber\\
&\ \ \
+X^a\mathcal{H}^d_a\dot{\partial}_d(A^c_b)Y^b-X^aY^dA^b_d\dot{\partial}_b\mathcal{H}^c_a+\dot{\tilde{X}}^a\dot{\partial}_a(Y^b)A^c_b
+\dot{\tilde{X}}^aY^b\dot{\partial}_a(A^c_b)\Big)\dot{\tilde{\partial}}_c.
\end{align}
(\ref{del4}) and (\ref{del5}) result in (\ref{Ex}).

Conversely, let (\ref{Ex}) holds. Suppose that $S$ and $\bar{S}$
have the locally expressions
\begin{align*}
S&=\left( g_{b}^{a}\circ h\circ \pi \right)
U^{b}\tilde{\partial}_{a}-2( G^{a}-\frac{1}{4}F^{a})
\dot{\tilde{\partial}}_{a},\\
\bar{S}&=\left( g_{b}^{a}\circ h\circ \pi \right)
U^{b}\tilde{\partial}_{a}-2( \bar{G}^{a}-\frac{1}{4}\bar{F}^{a})
\dot{\tilde{\partial}}_{a}.
\end{align*}
In the adapted basis $\{\bar{\tilde{\delta}}_a, \dot{\partial}_a\}$,
the $(\rho, \eta)$-spray $\bar{S}$ has the locally expression
\begin{align*}
\bar{S}=(g_{b}^{a}\circ h\circ \pi) U^{b}\bar{\tilde{\delta}}_a-(2(
\bar{G}^{a}-\frac{1}{4}\bar{F}^{a})+\bar{\mathcal{H}}^a_b(g_{c}^{b}\circ
h\circ \pi) U^{c}) \dot{\tilde{\partial}}_{a}.
\end{align*}
Now let $\bar{\mathcal{H}}^a_b=\mathcal{H}^a_b+B^a_b$. Then
$\bar{\tilde{\delta}}_a=\tilde{\delta}_a+B^b_a\dot{\tilde{\partial}}_b$,
where $\{\tilde{\delta}_a, \dot{\partial}_a\}$ is the adapted basis
induced by $\mathcal{H}_S$. Therefore in the adapted basis
$\{\tilde{\delta}_a, \dot{\partial}_a\}$, the $(\rho, \eta)$-spray
$\bar{S}$ has the locally expression
\begin{align*}
\bar{S}=(g_{b}^{a}\circ h\circ \pi) U^{b}\tilde{\delta}_a-(2(
\bar{G}^{a}-\frac{1}{4}\bar{F}^{a})+\mathcal{H}^a_b(g_{c}^{b}\circ
h\circ \pi) U^{c}) \dot{\tilde{\partial}}_{a}.
\end{align*}
Setting $X=\bar{S}$ and
$Y=\bar{U}=U^b\bar{\tilde{\delta}}_a+U^b\dot{\tilde{\partial}}_{b}=U^b\tilde{\delta}_a+(U^b+U^cB_c^b)\dot{\tilde{\partial}}_{b}$
in (\ref{Ex}) we get
\begin{align*}
0&=(\rho, \eta)\bar{\nabla}_{\bar{S}}\bar{U}=(\rho,
\eta)\nabla_{\bar{S}}\bar{U}-g^a_mU^mU^b(\dot{\partial}_bA^c_a)\tilde{\delta}_c+\Big(-g^a_dU^dU^bA^c_m\dot{\partial}_bH^m_a\\
&\ \ \ -g^a_dU^dU^bA^c_m\dot{\partial}_bA^m_a-g^a_dU^dU^b\dot{\partial}_bA^c_a-g^a_dU^dA^m_aU^n\dot{\partial}_mA^c_n\\
&\ \ \
-U^dg^a_mU^m\rho^i_a\partial_iA^c_d+g^a_mU^mU^dA^b_d\dot{\partial}_bH^c_a\\
&\ \ \
+U^d\dot{\partial}_aA^c_d\{2(\bar{G}^a-\frac{1}{4}\bar{F}^a)+B^a_bg^b_fU^f\}\Big)\dot{\tilde{\partial}}_c,
\end{align*}
Since $\mathcal{H}_S(\tilde{\delta}_a)=\tilde{\delta}_a$ and
$\mathcal{H}_S(\dot{\tilde{\partial}}_{a})=0$, then the above
equation gives 
\begin{align}\label{N1}
\mathcal{H}_S((\rho, \eta)\nabla_{\bar{S}}\bar{U})-(g^a_f\circ
h\circ\pi)U^fU^b(\dot{\partial}_bA^c_a)\tilde{\delta}_c=0.
\end{align}
But 
\begin{align*}
\mathcal{H}_S((\rho,
\eta)\nabla_{\bar{S}}\bar{U})&=\mathcal{H}_S\Big((g^a_d\circ
h\circ\pi)U^d\{U^b(\rho,
\eta)\nabla_{\tilde{\delta}_a}\tilde{\delta}_b+\tilde{\delta}_a(U^b)\tilde{\delta}_b\\
&\ \ \ +(U^b+U^eB^b_e)(\rho,
\eta)\nabla_{\tilde{\delta}_a}\dot{\tilde{\partial}}_b+\tilde{\delta}_a(U^b+U^eB^b_e)\dot{\tilde{\partial}}_b\}\\
&\ \ \ -(2(
\bar{G}^{a}-\frac{1}{4}\bar{F}^{a})+\mathcal{H}^a_d(g_{c}^{d}\circ
h\circ \pi)
U^{c})\{\dot{\partial}_a(U^b)\tilde{\delta}_b\\
&\ \ \
+\dot{\partial}_a(U^b+U^eB^b_e)\dot{\tilde{\partial}}_b\}\Big).
\end{align*}
Since
\begin{align*}
\mathcal{H}_S(\tilde{\delta}_a)&=\tilde{\delta}_a,\ \
\mathcal{H}_S((\rho,
\eta)\nabla_{\tilde{\delta}_a}\tilde{\delta}_b)=(\rho,
\eta)\nabla_{\tilde{\delta}_a}\tilde{\delta}_b=-\dot{\partial}_bH_a^c\tilde{\delta}_c,\\
\mathcal{H}_S(\dot{\tilde{\partial}}_{a})&=0,\ \ \
\mathcal{H}_S((\rho,
\eta)\nabla_{\tilde{\delta}_a}\dot{\tilde{\partial}}_b)=0, \ \
\tilde{\delta}_a(U^b)=\mathcal{H}^b_a,
\end{align*}
 then the above
equation reduces to
\begin{align*}
\mathcal{H}_S((\rho, \eta)\nabla_{\bar{S}}\bar{U})&=(g^a_d\circ
h\circ\pi)U^d\{-U^b\dot{\partial}_b\mathcal{H}_a^c\tilde{\delta}_c+\mathcal{H}^b_a\tilde{\delta}_b\}\\
&\ \ \ -(2(
\bar{G}^{b}-\frac{1}{4}\bar{F}^{b})+\mathcal{H}^b_d(g_{c}^{d}\circ
h\circ \pi) U^{c})\tilde{\delta}_b\\
&=-((g^a_d\circ h\circ\pi)U^dU^c\dot{\partial}_c\mathcal{H}_a^b+2(
\bar{G}^{b}-\frac{1}{4}\bar{F}^{b}))\tilde{\delta}_b.
\end{align*}
Setting $U^c\dot{\partial}_c\mathcal{H}_a^b=\mathcal{H}_a^b$ in the
above equation implies that
\begin{align*}
\mathcal{H}_S((\rho, \eta)\nabla_{\bar{S}}\bar{U})=-((g^a_d\circ
h\circ\pi)U^d\mathcal{H}_a^b+2(
\bar{G}^{b}-\frac{1}{4}\bar{F}^{b}))\tilde{\delta}_b.
\end{align*}
On the other hand in the proof of Lemma \ref{Important}, it had shown that
$(g^a_d\circ h\circ\pi)U^d\mathcal{H}_a^b=-2( G^b-\frac{1}{4}F^b)$.
Thus the above equation reduces to
\begin{align}\label{N2}
\mathcal{H}_S((\rho, \eta)\nabla_{\bar{S}}\bar{U})=2((
G^b-\frac{1}{4}F^b)-(
\bar{G}^{b}-\frac{1}{4}\bar{F}^{b}))\tilde{\delta}_b.
\end{align}
Also it is easy to see that
\begin{align}\label{N3}
(g^a_f\circ h\circ\pi)U^fU^b(\dot{\partial}_bA^c_a)=fU^c.
\end{align}
Putting (\ref{N2}) and (\ref{N3}) in (\ref{N1}), it results in
\[
-2(\bar{G}^b-\frac{1}{4}\bar{F}^b)=fU^b-2(G^{b}-\frac{1}{4}F^{b}).
\]
\end{proof}

\noindent
Constantin M Arcu\c{s}\\
Secondary School "Cornelius Radu"\\
Radinesti Village, 217196\\
Gorj County, Romania\\
Email:\ c\_arcus@radinesti.ro\\

\noindent
Esmail Peyghan and Esa Sharahi\\
Department of Mathematics, Faculty  of Science\\
Arak University\\
Arak 38156-8-8349,  Iran\\
Email: e-peyghan@araku.ac.ir,\ e-sharahi@phd.araku.ac.ir

\end{document}